\documentclass[11pt,reqno]{amsart}
\usepackage{amsthm, amssymb}
\usepackage[margin=1.1in, marginparwidth=1.8cm, marginparsep=0.3cm]{geometry}

\usepackage{hyperref}
\numberwithin{equation}{section}

\allowdisplaybreaks

\let\Re=\undefined\DeclareMathOperator*{\Re}{Re}
\let\Im=\undefined\DeclareMathOperator*{\Im}{Im}

\newcommand{\R}{\mathbb{R}}
\newcommand{\C}{\mathbb{C}}
\newcommand{\wh}[1]{\widehat{#1}}
\newcommand{\eps}{\varepsilon}

\newcommand{\M}{M}
\newcommand{\D}{{D}}
\newcommand{\F}{\mathcal{F}}

\newcommand{\who}[1]{\widehat{\overline{#1}}}
\renewcommand{\b}{\bar}

\newcommand{\ddd}{\,d\sigma\,d\eta\,ds}
\newcommand{\dd}{\,d\sigma\,d\eta}

\newcommand{\an}{\,\vert\partial_x\vert^{-1}}
\newcommand{\aan}{\,\vert\partial_x\vert^{-2}}

\newtheorem{theorem}{Theorem}[section]

\newtheorem{lemma}[theorem]{Lemma}

\newtheorem{proposition}[theorem]{Proposition}

\theoremstyle{definition}

\newtheorem{remark}[theorem]{Remark}

\theoremstyle{remark}
\newtheorem*{remarks}{Remarks}

\title[Cubic NLS]{Almost global existence for cubic nonlinear Schr\"odinger equations in one space dimension}

\author[J. Murphy]{Jason Murphy}
\address{Department of Mathematics, UC Berkeley}
\email{murphy@math.berkeley.edu}

\author[F. Pusateri]{Fabio Pusateri}
\address{Department of Mathematics, Princeton University}
\email{fabiop@math.princeton.edu}

\begin{document}

\date{\today}

\begin{abstract}
We consider non-gauge-invariant cubic nonlinear Schr\"odinger equations in one space dimension.
We show that initial data of size $\eps$ in a weighted Sobolev space lead to solutions with sharp $L_x^\infty$~decay up to time $\exp(C\eps^{-2})$.  We also exhibit norm growth beyond this time for a specific choice of nonlinearity.
\end{abstract}

\maketitle


			\section{Introduction}
			\label{sec:intro}
		
We study the initial-value problem for the following cubic nonlinear Schr\"odinger equation (NLS) in one space dimension:
			\begin{equation}
			\label{nls}			
			\begin{cases}
	(i\partial_t+\tfrac12\partial_{xx})u=
			\lambda_1\bar{u}^3+\lambda_2u^3+\lambda_3\vert u\vert^2\bar{u}+\lambda_4\vert u\vert^2 u,
			\\ u(1)=u_1\in\Sigma,
			\end{cases}
			\end{equation}
where $u:\R_t\times\R_x\to\C$ is a complex-valued function of space-time, $\lambda_j \in \C$ for $j=1,\dots4$, and
$$ \|u_1\|_{\Sigma}:=\|u_1\|_{L_x^2}+\|\partial_x u_1\|_{L_x^2}+\|xu_1\|_{L_x^2}.$$
Our main result, Theorem \ref{thm:main} below, is almost global existence for small solutions to \eqref{nls}.

The most widely-studied cubic NLS is the gauge-invariant equation
\begin{equation}\label{gauge-nls}
(i\partial_t+\tfrac12\partial_{xx})u=\pm\vert u\vert^2 u.
\end{equation}
{Gauge invariance} (that is, the symmetry $u\mapsto e^{i\theta}u$ for $\theta\in\R$) corresponds to the conservation of the $L_x^2$-norm. As the cubic NLS in one dimension is $L_x^2$-subcritical, 
this conservation law (together with Strichartz estimates) leads to a simple proof of global well-posedness of \eqref{gauge-nls} in $L_x^2$. 
As for long-time behavior, equation \eqref{gauge-nls} in one dimension is a borderline case for the $L_x^2$ scattering theory: 
for `short-range' nonlinearities $\vert u\vert^p u$ with $p>2$ there are positive results, 
while for the `long-range' case $0<p\leq 2$ there is no $L_x^2$ scattering \cite{Barab, TsutsumiYajima}. 
For \eqref{gauge-nls}, small data in $\Sigma$ lead to global solutions that decay in $L_x^\infty$ 
at the sharp rate $t^{-1/2}$ and exhibit modified scattering, that is, linear behavior up to a logarithmic phase correction as $t\to\infty$ 
(see for example \cite{DeiftZhou, HN0, KatoPusateri, IfrimTataru}).

For non-gauge-invariant equations like \eqref{nls}, the question of global existence is less well-understood.  Hayashi--Naumkin have studied non-gauge-invariant cubic NLS in one space dimension extensively
(see \cite{HN1, HN2, HN3, HN4, HN6, HN7, Naumkin}, for example).
They have shown that under specific conditions on the initial data, small data in weighted Sobolev spaces can lead to global solutions.
In these cases they are also able to describe the asymptotic behavior.  In this paper, we prove almost global existence for small (but otherwise arbitrary) data in $\Sigma$.  Such a result is in the spirit of the well-known works concerning quadratic wave equations in three dimensions \cite{John1, JK}.
Similar results have also been established for the cubic NLS with derivative nonlinearities; see, for example, \cite{SagawaSunagawa, Sunagawa}.
However, as mentioned in \cite{HN1, Naumkin, SagawaSunagawa},
there is a sense in which cubic nonlinearities containing at least one derivative can be considered `short-range',
while this is not the case for the problem without derivatives. We will discuss this in a bit more detail below in Section~\ref{sec:proof}.

Our main theorem is the following.
\begin{theorem}[Almost global existence]	\label{thm:main}
Let $u_1\in\Sigma$ and let $\eps:=\|u_1\|_\Sigma.$ If $\eps>0$ is sufficiently small, then there exists a unique solution
	$u\in C([1,T_\eps];\Sigma)$
to \eqref{nls} with $u(1)=u_1,$ where $T_\eps=\exp(\tfrac{1}{c\eps^2})$ for some absolute constant $c>0$. Furthermore,
for some $C>0$,
\begin{equation}\label{worse-J}
\sup_{t\in[1,T_\eps]}\big\{ \|\wh{u}(t)\|_{L_\xi^\infty} + t^{\frac12}\|u(t)\|_{L_x^\infty}+t^{-\frac14}\|(x+it\partial_x)u(t)\|_{L_x^2}\big\} \leq C\eps.
\end{equation}
\end{theorem}

In \eqref{nls}
we have decided to set the initial time $t_0=1$ for notational convenience. One could of course take $t_0=0$ with minor modifications.
Moreover, the same result also holds for negative times $t \in [-T_\eps,0]$.

It is important to observe that without imposing further conditions on the initial data or the coefficients in the nonlinearity,
Theorem~\ref{thm:main} is essentially sharp. To demonstrate this, we consider the particular model
\begin{equation}\label{model}
(i\partial_t+\tfrac12\partial_{xx})u = i|u|^2 u
\end{equation}
and show that solutions either blow up or exhibit norm growth after the almost global existence time.
The idea is that for sufficiently small data, certain ODE dynamics will dictate the behavior of the solution.
The particular model \eqref{model} has the following advantages:
(i) solutions to the ODE blow up in finite time,
and (ii) since $|u|^2 u$ is gauge-invariant, we get better estimates for $u$ than those appearing in Theorem~\ref{thm:main},
specifically, a slower growth rate for the $L_x^2$-norm of $(x+it\partial_x)u$.
Thanks to (i) we need not fine-tune the initial conditions to make the arguments work,
while (ii) allows us to show in a fairly straightforward fashion that the ODE can accurately model the PDE for long times.
The precise result we prove is the following.

\begin{theorem}[Norm growth]\label{thm:grow} There exists $\eps>0$ sufficiently small that the following holds.
Suppose $u_1\in\Sigma$ satisfies
\begin{equation}\label{asmp:u1}
\|u_1\|_{\Sigma} = \eps, \quad \|\wh{u_1}\|_{L_\xi^\infty}\geq\tfrac{1}{2}\eps,
\end{equation}
and let $u\in C([1,T_{\eps}];\Sigma)$ be the solution to \eqref{model} with $u(1)=u_1$ given by Theorem~\ref{thm:main}. In particular, $T_\eps=\exp(\tfrac{1}{c\eps^2})$ for some $c>0$, and
\[
\sup_{t\in[1,T_\eps]} \|\wh{u}(t)\|_{L_\xi^\infty}\lesssim\eps.
\]
Denoting by $T_{\max}\in(T_\eps,\infty]$ the maximal time of existence, there exists an absolute constant $K\gg\eps^2$
and a finite time $T_K> T_{\eps}$ such that either $T_{\max}\leq T_K$, or there exists $t\in [T_\eps,T_K]$ such that
\[
\|\wh{u}(t)\|_{L_\xi^\infty}^2 \geq K.
\]
\end{theorem}

\begin{remarks}\text{ }
\setlength{\leftmargini}{1.5em}
\begin{itemize}
\item The proof will show that we could take, for example, $K = (200c)^{-1}$.
This means that $K$ is a small but fixed constant independent from $\eps$ and, in particular, large compared to $\eps^2$.

\item The time $T_K$ is the time at which the associated ODE solution reaches size $4K$; see \eqref{def:TK}.

\item By the standard local theory (see below), if $T_{\max}<\infty$ then $\|u(t)\|_{L_x^2}\to\infty$ as $t\to T_{\max}$.

\item With trivial modifications, our arguments apply to \eqref{model} with a nonlinearity of the form $\lambda |u|^2 u$ with $\Im \lambda > 0$.\footnote{The case $\Im \lambda = 0$ reduces to \eqref{gauge-nls}, while for $\Im \lambda <0$
one can prove that small solutions exist on $[1,\infty)$ and have `dissipative' behavior, namely, additional logarithmic time decay \cite{Shimomura}.  This same dissipative behavior occurs for \eqref{model} in the negative time direction.}
\end{itemize}
\end{remarks}

The strategy described above, namely, deducing behavior about solutions from associated ODE dynamics, has been carried out in many previous works.  In the case of NLS with the $|u|^2 u$ nonlinearity, this approach leads to a proof of modified scattering \cite{HN0, KatoPusateri, IfrimTataru}.
In other cases, for some specific nonlinearities and well-prepared initial data one can prove global existence 
and describe the asymptotics \cite{HN1, HN2, HN3, HN4, HN6, HN7, Naumkin}.
In our case, we pick an equation for which the ODE solutions blow up; accordingly, we can demonstrate norm growth.
This example demonstrates that one cannot hope to improve on Theorem~\ref{thm:main} without imposing some more specific conditions.

	\subsection{Strategy of the proof of Theorem~\ref{thm:main}}
	\label{sec:proof}
We begin by recalling the standard local theory for \eqref{nls}.
			\begin{theorem}
			[Local well-posedness]
			\label{thm:lwp}
For $u_1\in L_x^2$ the initial-value problem
			\begin{equation}
			\nonumber
			\left\{\begin{array}{ll}
			(i\partial_t+\tfrac12\partial_{xx})u=
			\lambda_1\bar{u}^3+\lambda_2u^3+\lambda_3\vert u\vert^2\bar{u}+\lambda_4\vert u\vert^2 u
			\\ u(1)=u_1
			\end{array}
			\right.
			\end{equation}
has a unique solution $u\in C([1,T]; L_x^2)$, with $T\sim 1+\|u_1\|_{L_x^2}^{-4}$, obeying
			\begin{equation}
			\label{eq:duh}
		u(t)=e^{i(t-1)\partial_{xx}/2}u_1-i\int_1^t
		e^{i(t-s)\partial_{xx}/2}[
		\lambda_1\bar{u}^3(s)+\lambda_2u^3(s)+\lambda_3\vert u\vert^2\bar{u}(s)+\lambda_4\vert u\vert^2 u(s)]\,ds.
			\end{equation}
Furthermore, if $u_1\in\Sigma$ then $u\in C([1,T];\Sigma).$
			\end{theorem}
The existence in $C([1,T];L_x^2(\R))$ follows from the standard arguments, namely contraction mapping and Strichartz estimates.
The fact that the time of existence depends only on the {norm} of the data is a consequence of scaling. The existence in $C([1,T];\Sigma)$ follows from standard persistence of regularity arguments, which involve commuting the equation with $\partial_x$ and $J(t)=x+it\partial_x$.  We refer the reader to the textbook \cite{Caz} and	 the references cited therein.

For a solution $u$ we define
$$f(t)=e^{-it\partial_{xx}/2}u(t), \qquad Ju(t)=(x+it\partial_x)u(t).$$
The proof of Theorem \ref{thm:main} will be based on a bootstrap argument in a properly chosen norm.  To this end we introduce the notation
		$$\|u(t)\|_{X(t)}:=
		\tfrac12\bigl[\|\wh{f}(t)\|_{L_\xi^\infty}+
		t^{-\frac14}\|Ju(t)\|_{L_x^2}\bigr].$$
We record here two facts that we prove in Section~\ref{sec:decay}.
		\begin{lemma}
		\label{lem:decay}
The following estimates hold:
		\begin{align}
		\label{eq:Xdata}
		\|u(1)\|_{X(1)}&\leq \|u_1\|_{\Sigma},
		\\
		\label{eq:Xdecay}
		t^{\frac12}\|u(t)\|_{L_x^\infty}
		&\lesssim
		\|u(t)\|_{X(t)}.
		\end{align}
		\end{lemma}

The next two propositions are the main ingredients for the bootstrap argument used to prove Theorem~\ref{thm:main}; they constitute the heart of the paper. 

		\begin{proposition}
		\label{prop:fhat}
For $u:[1,T]\times\R\to\C$ a solution to \eqref{nls} and $1\leq t\leq T$, there exists an absolute constant $C>0$ such that
		\begin{align}
		\label{eq:fhat bootstrap}
		\|\wh{f}(t)\|_{L_\xi^\infty} 
		& \leq \|\wh{f}(1)\|_{L^\infty_\xi} + C\Big[
		\|u_1\|_{\Sigma}^3 + \|u(t)\|_{X(t)}^3 + \int_1^t s^{-1}\big(\|u(s)\|_{X(s)}^3+\|u(s)\|_{X(s)}^5\big)\,ds \Big].
		\end{align}
		\end{proposition}

		\begin{proposition}
		\label{prop:Ju}
For $u:[1,T]\times\R\to\C$ a solution to \eqref{nls} and $1\leq t\leq T$,  there exists an absolute constant $C>0$ such that
		\begin{align}
		\label{eq:Ju bootstrap}
		\|Ju(t)\|_{L_x^2}
		& \leq \|J u(1)\|_{L_x^2} + C\Bigl[ \|u_1\|_{\Sigma}^3 + t^{\frac14}\|u(t)\|_{X(t)}^3
		+ \int_1^t s^{-\frac34}\big(\|u(s)\|_{X(s)}^3 + \|u(s)\|_{X(s)}^5\big)\,ds \Bigr].
		\end{align}
		\end{proposition}

We prove Propositions~\ref{prop:fhat}~and~\ref{prop:Ju} in Sections~\ref{sec:fhat}--\ref{sec:Ju}
by performing an analysis in Fourier space known as the space-time resonance method \cite{GMS3D, GMS2D}.
More precisely, we begin by looking at the integral equation \eqref{eq:duh} and expressing it in terms of the profile $f = e^{-it\partial_{xx}/2}u$
and in Fourier space as in \eqref{eq:duhamel}--\eqref{Phases1}.
We do not follow this approach for the gauge-invariant term $\vert u\vert^2 u$, since it is amenable to a simpler treatment altogether, which in particular does not necessitate analysis via space-time resonance. 

We then proceed to study the oscillations in the integrals \eqref{eq:duhamel}.
The most delicate interactions arise when there is a lack of oscillation in $(\eta,\sigma,s)$,
that is, when the phases in \eqref{Phases1} vanish together with their gradients in $\eta$ and $\sigma$.
The region of $(\eta,\sigma)$ in $\R^2$ where this vanishing occurs is known as the space-time resonant set.

For the three non-gauge-invariant cubic nonlinearities, the space-time resonant set is the origin.  
To deal with the contribution of this set, our strategy is to introduce a time-dependent cutoff to a neighborhood of the origin where we use volume bounds.
We then decompose the complement of this neighborhood into regions where we can integrate by parts in either space (in $\eta$ or $\sigma$) or time (in $s$), using the identities
\[
e^{isA}=(is\partial_\eta A)^{-1} \partial_\eta e^{isA}, \quad e^{isA}=(is\partial_\sigma A)^{-1} \partial_\sigma e^{isA}, \quad e^{isA}=(iA)^{-1} \partial_s e^{isA},
\]
respectively, where $A$ is one of the phases appearing in \eqref{eq:duhamel}--\eqref{Phases1}.  This procedure yields additional decay either by introducing the factor $s^{-1}$ or by introducing more copies of the solution (cf. \eqref{eq:Xdecay} and the fact that $\partial_s f = e^{-is\partial_{xx}/2}(\partial_s +\tfrac{i}{2}\partial_{xx})u$ is a cubic expression in $u$). 

Thanks to our decompositions of the frequency space, the multipliers of the form $(\partial_\eta A)^{-1}$ or $A^{-1}$ that appear 
after the integration by parts can be viewed as (powers of) antiderivatives acting on the highest frequency terms, up to multiplication by Coifman--Meyer multipliers.  We point out that that the contribution of the term $\bar{u}^3$ is the easiest to
estimate, since away from the origin we have complete temporal non-resonance (the phase $\Phi$ in \eqref{Phases1} is bounded below).  For $u^3$ and $\vert u\vert^2\bar{u}$, we need to decompose the frequency space more carefully.
The use of the Coifman--Meyer Theorem (see Lemma \ref{thm:cm} below) is crucial for our arguments, since it gives the sharp H\"older-type estimates that allow us to prove optimal lifespan bounds.

Note that from the perspective of space-time resonance, the presence of derivatives in the nonlinearity 
actually offers some improvement compared to the nonlinearities we consider in \eqref{nls}.
Indeed, derivatives act as multiplication by the frequency on the Fourier side and hence provide some cancellation at zero frequency, that is, on the space-time resonant set. In particular, this can be thought of as a type of null condition (see \cite{PS}, for example).
We refer the reader especially to \cite{GMS2D}, which employs the space-time resonance method to prove global existence and scattering for a non-gauge-invariant quadratic NLS in two space dimensions, with a nonlinearity containing a derivative at low frequencies.

Assuming Propositions~\ref{prop:fhat}~and~\ref{prop:Ju} for now, we prove Theorem~\ref{thm:main}.
			
\begin{proof}[Proof of Theorem~\ref{thm:main}]
Let $0<\eps<1$ to be specified below and let $\|u_1\|_{\Sigma} = \eps.$  If $u$ solves \eqref{nls}, then Proposition~\ref{prop:fhat}, Proposition~\ref{prop:Ju}, and Lemma~\ref{lem:decay} imply
		\begin{align}		
		\|u(t)\|_{X(t)}
		& \leq \eps + C\Big[ 2\|u_1\|_{\Sigma}^3 + \|u(t)\|_{X(t)}^3
		+ \! \int_1^t (s^{-1}\!+t^{-\frac14}s^{-\frac34})
		\big( \|u(s)\|_{X(s)}^3+\|u(s)\|_{X(s)}^5 \big) \,ds \Big]
		\label{eq:proof1}
		\end{align}	
for some absolute constant $C>0$.  We choose $\eps=\eps(C)>0$ and define $T_\eps$ so that
		\begin{align}
		\label{eq:proof2}	
		170C\eps^2 < \tfrac{1}{2}, \qquad T_\eps:=\exp\big(\tfrac{1}{80C\eps^2}\big).
		\end{align}
We now claim that the following estimate holds:
		\begin{equation}
		\|u(t)\|_{X(t)}\leq 2\eps \quad\text{for all}\quad t\in[1,T_\eps].
		\label{eq:proof4}	
		\end{equation}
This holds at $t=1$ by \eqref{eq:Xdata}.  By continuity, if it is not true for all $t\in[1,T_\eps]$ there must be a first time $t\in(1,T_\eps]$ such that $\|u(t)\|_{X(t)}=2\eps$. 
Applying \eqref{eq:proof1} at this time and using \eqref{eq:proof2} yields
		\begin{align*}
		2\eps
		&\leq 
		\eps + C\big(2\eps^3 + (2\eps)^3 + [4+\log t] [(2\eps)^3+(2\eps)^5] \big)
		\\
		&\leq
		\eps \big(1 + 10C\eps^2 + C[4 + \log T_\eps] [ 40 \eps^2] \big)
		< 2\eps,
		\end{align*}
which is a contradiction. This proves \eqref{eq:proof4}.

To complete the proof, it suffices to show that if $u:[1,T]\times\R\to\C$ is a solution such that $T\leq \exp\big(\frac{1}{c\eps^2}\big)$
and $\sup_{t\in[1,T]}\|u(t)\|_{X(t)}\lesssim \eps,$ then we may continue the solution in time.
By the local theory it suffices to prove that $\|u(T)\|_{L_x^2}\lesssim \|u_1\|_{L_x^2}.$
We use the Duhamel formula \eqref{eq:duh}, Lemma~\ref{lem:decay}, and the bound on $u$ to estimate
		\begin{align*}
		\|u(T)\|_{L_x^2}
		&\lesssim \|u_1\|_{L_x^2}+\!\int_1^T
		\|u(s)\|_{L_x^\infty}^2\|u(s)\|_{L_x^2}\,ds
		\lesssim \|u_1\|_{L_x^2}+\eps^2\!\int_1^T s^{-1}
		\|u(s)\|_{L_x^2}\,ds.
		\end{align*}
Thus by Gronwall's inequality and the bound on $T$, we deduce $\|u(T)\|_{L_x^2}\lesssim T^{C\eps^2}\|u_1\|_{L_x^2} \lesssim \|u_1\|_{L_x^2},$ as was needed to show. This completes the proof of Theorem~\ref{thm:main}. \end{proof}

The rest of the paper is organized as follows: In Section~\ref{sec:basic} we set up notation and collect some useful lemmas.  The main trilinear estimates that we will use repeatedly in the proofs of Proposition~\ref{prop:fhat}~and~\ref{prop:Ju} are given in Lemma \ref{lem:tri}.
In Section~\ref{sec:fhat} we prove Proposition~\ref{prop:fhat}, and in Section~\ref{sec:Ju} we prove Proposition~\ref{prop:Ju}.  As shown above, these two propositions imply the main result, Theorem~\ref{thm:main}.  Section~\ref{sec:grow} contains the proof of Theorem~\ref{thm:grow}, in which we demonstrate norm growth for a model nonlinearity.  In Appendix~\ref{section:appendix} we discuss the construction of some cutoffs used in Sections~\ref{sec:fhat}~and~\ref{sec:Ju}.

		\subsection*{Acknowledgements}

J.~M. was supported by the NSF Postdoctoral Fellowship DMS-1400706.
F.~P. was supported in part by NSF grant DMS-1265875.


		\section{Notation and useful lemmas}
		\label{sec:basic}

For nonnegative $X,Y$ we write $X\lesssim Y$ to denote $X\leq CY$ for some $C>0$. We write $X\ll Y$ to denote $X\leq cY$ for some small $c \in (0,1)$.
We write $\text{\O}(X)$ to denote a finite linear combination of terms that resemble $X$ up to constants, complex conjugation,
and Littlewood--Paley projections. For example, the nonlinearity in \eqref{nls} is $\text{\O}(u^3)$.
	
The Fourier transform and its inverse are given by
		$$
		\F u(\xi)=\wh{u}(\xi)=
		(2\pi)^{-\frac12}\int_\R e^{-ix\xi}u(x)\,dx,
		\quad
		\F^{-1}u(x)=
		(2\pi)^{-\frac12}\int_\R e^{ix\xi} u(\xi)\,d\xi.
		$$
		
	For $s\in\R$ we define the fractional derivative operator $\vert\partial_x\vert^s$ as a Fourier multiplier, namely, $\vert\partial_x\vert^s=\F^{-1}\vert\xi\vert^s\F.$ We define the homogeneous Sobolev space $\dot{H}_x^s$ via
		$$\|u\|_{\dot{H}_x^s}=\|\,\vert\partial_x\vert^s u\|_{L_x^2}.$$
	
We employ the standard Littlewood--Paley theory. Let $\phi:\R\to\R$ be an even bump supported on $[-\frac{10}{9},\frac{10}{9}]$ and equal to one on $[-1,1]$.
For $N \in 2^{\mathbb{Z}}$ we define
		\begin{align*}
		&\wh{P_{\leq N}f}(\xi)=\wh{f_{\leq N}}(\xi) :=\phi(\xi/N)\wh{f}(\xi),
		\quad \wh{P_{>N} f}(\xi)=\wh{f_{>N}}(\xi) :=[1-\phi(\xi/N)]\wh{f}(\xi),
		\\ &\wh{P_N f}(\xi)=\wh{f_N}(\xi)
		:=[\phi(\xi/N)-\phi(2\xi/N)]\wh{f}(\xi).
		\end{align*}
These operators commute with all other Fourier multiplier operators.
They are self-adjoint and bounded on every $L_x^p$-space and obey the estimate
\begin{equation}
\label{lem:bernstein}
\|f_{>N}\|_{L_x^q(\R)}  \lesssim N^{-s+\frac{1}{r}-\frac{1}{q}}\|\,\vert\partial_x\vert^s f_{>N}\|_{L_x^r(\R)}
\end{equation}
for $1\leq r\leq q\leq\infty$ and $s>\tfrac{1}{r}-\tfrac{1}{q}$.


		

\subsection{Linear theory} \label{sec:linear}
		
The free Schr\"odinger propagator is defined as a Fourier multiplier: $e^{it\partial_{xx}/2}=\F^{-1} e^{-it\xi^2/2}\F.$
In physical space we have	
		\begin{equation}
		\label{eq:propagator}   
		[e^{it\partial_{xx}/2}f](x)=
		(2\pi i t)^{-\frac12}\int_\R
		e^{i(x-y)^2/2t}f(y)\,dy.
		\end{equation}
From \eqref{eq:propagator} we can read off the following factorization:
		\begin{equation}
		\label{eq:factorization}
		e^{it\partial_{xx}/2}=
		\M(t)\D(t)\F\M(t),
		\end{equation}
where the modulation $M(t)$ and dilation $\D(t)$ are defined by
		$$
		[M(t)f](x)=e^{ix^2/2t}f(x)
		\quad
		\text{and}
		\quad
		[\D(t)f](x)=(it)^{-\frac12}f(\tfrac{x}{t}).
		$$		

We define the operator $J(t)=x+it\partial_x$. By \eqref{eq:factorization}, we have $J(t)=e^{it\partial_{xx}/2}xe^{-it\partial_{xx}/2}.$
For a solution $u(t)$ to \eqref{nls} we write $f(t)=e^{-it\partial_{xx}/2}u(t)$ and note that
		\begin{equation}
		\label{eq:forms of Ju} 	
		\|Ju\|_{L_x^2}
		= \|xf\|_{L_x^2}
		= \|\partial_\xi \wh{f}\|_{L_\xi^2}.
		\end{equation}
		

\subsection{Notation and Duhamel formula}\label{sec:notation}
Suppose that $u$ is a solution to \eqref{nls} and denote $f(t)=e^{-it\partial_{xx}/2}u(t)$.
Using the Duhamel formula \eqref{eq:duh} and taking the Fourier transform leads to:
\begin{align}
\label{eq:duhamel} 
\begin{split}
	\wh{f}(t,\xi)
	&=
	\wh{f}(1,\xi)
	-i\lambda_1(2\pi)^{-1} \int_1^t \iint_{\R^2} e^{is\Phi(\xi,\eta,\sigma)}
	\who{f}(\xi-\eta)\who{f}(\eta-\sigma)\who{f}(\sigma)\ddd
	\\ &\quad
	-i\lambda_2(2\pi)^{-1} \int_1^t\iint_{\R^2} e^{is\Psi(\xi,\eta,\sigma)}
	\wh{f}(\xi-\eta)\wh{f}(\eta-\sigma)\wh{f}(\sigma)\ddd
	\\ &\quad-i\lambda_3(2\pi)^{-1} \int_1^t\iint_{\R^2} e^{is\Omega(\xi,\eta,\sigma)}
		\who{f}(\xi-\eta)\who{f}(\eta-\sigma)\wh{f}(\sigma)\ddd
	\\ &\quad -i\lambda_4 \int_1^t\F\big(e^{-is\partial_{xx}/2}\vert u\vert^2u\big)(s,\xi)\,ds,
\end{split}
\end{align}
where the phases $\Phi,$ $\Psi$, and $\Omega$ are given by
	\begin{align}
	 \begin{split}
	  \label{Phases1}
	&\Phi=\tfrac12[\xi^2+(\xi-\eta)^2
		+(\eta-\sigma)^2+\sigma^2],
	\\ &\Psi=\tfrac12[\xi^2-(\xi-\eta)^2
		-(\eta-\sigma)^2-\sigma^2],
	\\ &\Omega=\tfrac12[\xi^2+(\xi-\eta)^2
		+(\eta-\sigma)^2-\sigma^2].
	 \end{split}
	\end{align}
We do not write out the phase for the gauge-invariant nonlinearity $\vert u\vert^2 u$, since this term is amenable to a simpler analysis.

	It is convenient to introduce the notation
\begin{align}
\label{notfreq}
\xi_1=\xi-\eta,\quad
		\xi_2=\eta-\sigma,\quad
		\xi_3=\sigma,\quad
		\vec{\xi}=(\xi_1,\xi_2,\xi_3),\quad
		\xi=\xi_1+\xi_2+\xi_3.
\end{align}
In this notation we may rewrite the phases as follows:
		\begin{align}
		 \begin{split}
		  \label{Phases2}
		&\Phi=\tfrac12[\xi^2+\xi_1^2+\xi_2^2+\xi_3^2],
		\quad \, \Psi =\xi_1\xi_2+\xi_2\xi_3+\xi_3\xi_1,
		\quad \, \Omega= \xi_1^2+\xi_2^2+\xi_1\xi_2+\xi_2\xi_3+\xi_1\xi_3.
		 \end{split}
		\end{align}
We also need to consider derivatives of the phases, which we record here:
	\begin{align}
	\begin{split}
	\label{Phases3}	
	&\partial_\xi\Phi=\xi+\xi_1,\quad
	\partial_\xi\Psi=\xi_2+\xi_3,\quad
	\partial_\xi\Omega=\xi+\xi_1,
	\\ &(\partial_\eta\Psi,\partial_\sigma\Psi)
	=(\xi_1-\xi_2,\xi_2-\xi_3),\quad
	(\partial_\eta\Omega,\partial_\sigma\Omega)=
	(\xi_2-\xi_1,-\xi_2-\xi_3).
	\end{split}
	\end{align}

Finally, we set up notation concerning frequency cutoffs.  For a function $f=f(s,x)$ and $s\geq 1$, we define $f_{lo}$ and $f_{hi}$ via $\wh{f_{lo}}=\wh{P_{lo}f}:=\phi_{lo}\wh{f}$ and $\wh{f_{hi}}=\wh{P_{hi}f}=:\phi_{hi}\wh{f}$, where
		\begin{equation}
		\label{eq:cutoff}	
		\phi_{lo}(\xi) :=\phi(s^{\frac12}\xi),
		\quad\phi_{hi}(\xi) :=1-\phi(s^{\frac12}\xi).
		\end{equation}
Here $\phi$ is the standard cutoff defined earlier at the beginning of this section. 
We use the notation $\phi_*$ to denote that either $\phi_{lo}$ or $\phi_{hi}$ may appear, and $f_*=P_*f$ to denote that either ${f_{lo}}$ or ${f_{hi}}$ may appear.

		\subsection{Proof of Lemma~\ref{lem:decay}}
		\label{sec:decay}
In this section we prove \eqref{eq:Xdata} and \eqref{eq:Xdecay}. First,
	\begin{align*}
	\|Ju_1\|_{L_x^2}&\leq \|xu_1\|_{L_x^2}+\|\partial_x u_1\|_{L_x^2},
	\\ \|\wh{u_1}\|_{L_x^\infty}&
	\leq (2\pi)^{-\frac12} \|u_1\|_{L_x^1}\leq \pi^{-\frac12}	\|(1+\vert x\vert)u_1\|_{L_x^2},
	\end{align*}
so that \eqref{eq:Xdata} holds. For \eqref{eq:Xdecay} we first use \eqref{eq:factorization} to write
		$$
		u(t)=\M(t)\D(t)\wh{f}(t)
			+\M(t)\D(t)\F[M(t)-1]f(t).
		$$
It therefore suffices to show
		$$
		\|\F[M(t)-1]f(t)\|_{L_x^\infty}
		\lesssim t^{-\frac14}\|Ju(t)\|_{L_x^2}.
		$$

	We split at frequency $\sqrt{t}$, using  the operators
		$
		\bar{P}_{\leq N}:=
		\F\phi(\cdot/N)\F^{-1}.
		$
These share the same estimates as the usual projections, as
		$
		\bar{P}_{\leq N}f=\overline{P_{\leq N}\bar{f}}.
		$

	We first use \eqref{lem:bernstein}, Plancherel, and \eqref{eq:forms of Ju} to estimate
		\begin{align*}
		\|\bar{P}_{>\sqrt{t}}\F[M(t)-1]f(t)\|_{L_x^\infty}
		 &\lesssim t^{-\frac14}\|\partial_x\big[
			\F[M(t)-1]f(t)\big]\|_{L_x^2}
		\\ &\lesssim t^{-\frac14}\|xf(t)\|_{L_x^2}
		\lesssim t^{-\frac14}\|Ju(t)\|_{L_x^2}.
		\end{align*}

	Second, we use Hausdorff--Young, Cauchy--Schwarz, \eqref{eq:forms of Ju}, and the bound
		\begin{equation}
		\label{eq:pw}
		\vert \M(t)-1\vert
		=\vert e^{ix^2/2t}-1\vert
		\lesssim t^{-\frac12}\vert x\vert
		\end{equation}
to estimate
		\begin{align*}
		\|\bar{P}_{\leq\sqrt{t}}
		\F[M(t)-1]f(t)\|_{L_x^\infty}&
		\lesssim \|\phi(\tfrac{\cdot}{\sqrt{t}})[M(t)-1]f(t)\|_{L_x^1}
		\lesssim t^{\frac14}\|[M(t)-1]f(t)\|_{L_x^2}
		\lesssim t^{-\frac14}\|xf(t)\|_{L_x^2}. 
		\end{align*}

		\subsection{Useful estimates}
		\label{sec:frequent}

For a function $m:\R^3\to\R$ we define the trilinear operator $T_m$ as follows: 	
		\begin{equation}
		\label{def:trilinear}  
		\F\big(T_m[a,b,c]\big)(\xi)=
		\iint_{\R^2} m(\xi-\eta,\eta-\sigma,\sigma)
		\wh{a}(\xi-\eta)
		\wh{b}(\eta-\sigma)
		\wh{c}(\sigma)\,d\sigma\,d\eta.
		\end{equation}
If $a,$ $b,$ $c$ are functions of space-time, we employ the following notation:
\[
\F\big(T_m[a(t),b(t),c(t)]\big)(\xi)=\F\big(T_m[a,b,c]\big)(t,\xi).
\]	
We also make use of the notation introduced in \eqref{notfreq}.

The following multilinear estimate due to Coifman--Meyer is one of the primary technical tools used in this paper.  For the original result, see \cite[Chapter~13]{CoifmanMeyer}; for a more modern treatment, see \cite{Muscalu}.
		\begin{lemma}
		[Coifman--Meyer estimate
		\cite{CoifmanMeyer}]
		\label{thm:cm}		
Let $m \in C^{\infty}(\R^3\backslash\{0\}; \R)$ be a symbol satisfying
	\begin{align}
	 \label{CM}
		\sup_{\vec{\xi}\in\R^3\backslash\{0\}}
		\big| \vec{\xi} \big|^{\vert\alpha\vert} \big| \partial_\xi^\alpha m(\vec{\xi}) \big| < \infty
	\end{align}
for all multiindices $\alpha$ with $|\alpha| \leq 10$. Then
		$$\|T_m[a,b,c]\|_{L_x^r}\lesssim
		 \|a\|_{L_x^{r_1}}
		 \|b\|_{L_x^{r_2}}
		 \|c\|_{L_x^{r_3}}$$
for all $1<r_1,r_2,r_3\leq\infty$ and $1\leq r<\infty$ such that $\tfrac{1}{r}=\tfrac{1}{r_1}+\tfrac{1}{r_2}+\tfrac{1}{r_3}.$
		\end{lemma}
\begin{remark} Symbols with the property \eqref{CM} will be called {\it Coifman--Meyer} symbols.\end{remark}

We will now establish some trilinear estimates that will be used frequently in Sections~\ref{sec:fhat}~and~\ref{sec:Ju}.  The proofs rely on Lemma~\ref{thm:cm}, together with the following estimate, which is a consequence of Plancherel and H\"older:
\begin{equation}
\label{est0}
N^{\frac12} \|u_{>N}\|_{\dot{H}_x^{-1}} + N^{\frac32} \|u_{>N}\|_{\dot{H}_x^{-2}} + N^{-\frac12}\|u_{\leq N}\|_{L_x^2} \lesssim \|\wh{u}\|_{L_\xi^\infty}.
\end{equation}

\begin{lemma}[Trilinear estimates]\label{lem:tri}
Let $T_m$ be an operator of the form \eqref{def:trilinear}.
\setlength{\leftmargini}{1.8em}
\begin{itemize}
\item[(i)] If $|\xi_3|^2 m$ is a Coifman--Meyer symbol supported where $|\xi_3| \gtrsim \max\{|\xi_2|,|\xi_1|\}$, then
\begin{align}
\label{TriEst1}
{\| \F\bigl({T_m}[a,b,c_{>N}]\bigr) \|}_{L^\infty_\xi} \lesssim N^{-1}  {\|\wh{a}\|}_{L^\infty_\xi} \min\big( {\|b\|}_{L_x^\infty} {\|\wh{c}\|}_{L^\infty_\xi},  {\|\wh{b}\|}_{L^\infty_\xi} {\|c\|}_{L_x^\infty} \big).
\end{align}

\item[(ii)] If $|\xi_3|m$ is a Coifman--Meyer symbol supported where $|\xi_3| \gtrsim \max\{|\xi_2|,|\xi_1|\}$, then
\begin{align}
\label{TriEst2}
{\| \F\bigl({T_m}[a,b,c_{>N}]\bigr) \|}_{L^\infty_\xi} \lesssim N^{-\frac12}  {\|a\|}_{L^\infty_x} \min\big( {\|b\|}_{L_x^2} {\|\wh{c}\|}_{L^\infty_\xi},
  {\|\wh{b}\|}_{L^\infty_\xi} {\|c\|}_{L_x^2} \big).
\end{align}
\item[(iii)] If $|\xi_3|m$ is a Coifman--Meyer symbol supported where $|\xi_3| \gtrsim \max\{|\xi_2|,|\xi_1|\}$, then
\begin{align}
\label{TriEst3}
{\| T_m[a,b,c_{>N}] \|}_{L_x^2} \lesssim N^{-\frac12}  {\|a\|}_{L^\infty_x} \min\big( {\|b\|}_{L_x^\infty} {\|\wh{c}\|}_{L^\infty_\xi},
  {\|\wh{b}\|}_{L^\infty_\xi} {\|c\|}_{L_x^\infty} \big).
\end{align}
\end{itemize}

In all of the above estimates, we can exchange the role of $a$ and $b$ on the right-hand side.
\end{lemma}

As explained above, we obtain \emph{a priori} bounds on solutions to \eqref{eq:duhamel} by analyzing the cubic terms in Fourier space.  This leads us to study trilinear expressions of the form \eqref{def:trilinear} with symbols that have some degeneracies for small frequencies.  By properly dividing frequency space, we will be able to show that these singularities are always of the form $\max(|\xi_1|,|\xi_2|,|\xi_3|)^{-1}$ or $\max(|\xi_1|,|\xi_2|,|\xi_3|)^{-2}$.  We will therefore be able to use the bounds \eqref{TriEst1}--\eqref{TriEst3} to control these expressions; in particular, we will make the choice $N\sim s^{-1/2}$, where $s$ is the time variable.

\begin{proof}[Proof of Lemma~\ref{lem:tri}]  The estimates \eqref{TriEst1}--\eqref{TriEst3} will all follow from Lemma~\ref{thm:cm} and \eqref{est0}, together with suitable high-low decompositions.

{\it Proof of \eqref{TriEst1}.}  Suppose $|\xi_3|^2 m$ is a Coifman--Meyer symbol supported in a region where $|\xi_3| \gtrsim \max\{|\xi_2|,|\xi_1|\}$.

We decompose $a=a_{\leq N}+a_{>N}$ and first estimate
\begin{align*}
{\| \F\bigl(T_m[a_{\leq N},b,c_{>N}]\bigr) \|}_{L^\infty_\xi} & \lesssim  {\|\F\bigl(T_{\xi_3^2m}[a_{\leq N},b, |\partial_x|^{-2} c_{>N}]\bigr) \|}_{L^\infty_\xi}
\\ &\lesssim {\|a_{\leq N} \|}_{L_x^2} {\| b \|}_{L_x^\infty} {\| c_{>N} \|}_{\dot H_x^{-2}}
\lesssim N^{-1}  {\|\wh{a}\|}_{L^\infty_\xi} {\|b\|}_{L_x^\infty} {\|\wh{c}\|}_{L_\xi^\infty}.
\end{align*}
Next, note that under our assumptions, $|\xi_1\xi_3|m$ is also a Coifman--Meyer multiplier. Thus,
\begin{align*}
{\| \F\bigl(T_m[a_{> N},b,c_{>N}]\bigr) \|}_{L^\infty_\xi} & \lesssim {\| \F\bigl(T_{|\xi_1\xi_3|m}[\an a_{> N}, b,\an c_{>N}]\bigr) \|}_{L^\infty_\xi}
\\
& \lesssim {\|a_{> N} \|}_{\dot H_x^{-1}} {\| b \|}_{L_x^\infty} {\|c_{>N} \|}_{\dot H_x^{-1}}
\lesssim N^{-1} {\|\wh{a} \|}_{L_\xi^\infty} {\|b\|}_{L_x^\infty}{\| \wh{c} \|}_{L_\xi^\infty}.
\end{align*}
Combining the two estimates above yields the first estimate in \eqref{TriEst1}.

We turn to the second inequality in \eqref{TriEst1}.  We decompose both $a=a_{\leq N}+a_{>N}$ and $b=b_{\leq N}+b_{>N}$.  Using \eqref{lem:bernstein} as well, we first have
\begin{align*}
{\| \F\bigl(T_m[a_{\leq N},b_{\leq N},c_{>N}]\bigr) \|}_{L^\infty_\xi} &\lesssim {\| \F\bigl(T_{\xi_3^2m}[a_{\leq N},b_{\leq N}, |\partial_x|^{-2} c_{>N}]\bigr) \|}_{L^\infty_\xi} \\
&\lesssim {\|a_{\leq N} \|}_{L_x^2} {\| b_{\leq N}\|}_{L_x^2} {\| \aan c_{>N} \|}_{L_x^\infty}
\lesssim N^{-1}  {\|\wh{a}\|}_{L^\infty_\xi} {\|\wh{b}\|}_{L^\infty_\xi} {\|c\|}_{L_x^\infty}.
\end{align*}
Next, note that under our assumptions, $\xi_2^2 m$ is also Coifman--Meyer.  Thus,
\begin{align*}
{\| \F\bigl(T_m[a_{\leq N},b_{> N},c_{>N}]\bigr) \|}_{L^\infty_\xi} &\lesssim {\| \F\bigl(T_{\xi_2^2m}[a_{\leq N}, |\partial_x|^{-2} b_{>N},c_{>N}]\bigr) \|}_{L^\infty_\xi}
\\ &\lesssim {\|a_{\leq N} \|}_{L_x^2} {\|b_{>N} \|}_{\dot H_x^{-2}} {\| c \|}_{L_x^\infty}
\lesssim N^{-1} {\|\wh{a} \|}_{L_\xi^\infty} {\|\wh{b} \|}_{L_\xi^\infty} {\| c \|}_{L_x^\infty}.
\end{align*}
Noting that $\xi_1^2m$ is also Coifman--Meyer, we can similarly obtain
\[
{\| \F\bigl(T_m[a_{> N},b_{\leq N},c_{>N}]\bigr) \|}_{L_\xi^\infty} \lesssim  N^{-1} {\|\wh{a} \|}_{L_\xi^\infty}{\|\wh{b} \|}_{L_\xi^\infty} {\| c \|}_{L_x^\infty}.
 \]
The remaining case can be treated similarly, as $|\xi_1\xi_2|m$ is also Coifman--Meyer.

{\it Proof of \eqref{TriEst2}.} Suppose  $|\xi_3| m$ is a Coifman--Meyer symbol supported in a region where $|\xi_3| \gtrsim \max\{|\xi_2|,|\xi_1|\}$.

We can obtain the first estimate in \eqref{TriEst2} as follows:
\begin{align*}
{\| \F\bigl(T_m[a,b,c_{>N}]\bigr) \|}_{L^\infty_\xi}& \lesssim {\| \F\bigl(T_{|\xi_3|m}[a,b,\an c_{>N}]\bigr) \|}_{L^\infty_\xi}
\\ &\lesssim {\|a\|}_{L_x^\infty} {\| b \|}_{L_x^2} {\|c_{>N} \|}_{\dot H_x^{-1}}
 \lesssim N^{-\frac12} {\|a\|}_{L_x^\infty} {\| b \|}_{L_x^2} {\| \wh{c} \|}_{L_\xi^\infty}.
\end{align*}

To obtain the second estimate in \eqref{TriEst2}, we decompose $b = b_{\leq N} + b_{> N}$. Using \eqref{lem:bernstein} as well, we first have
\begin{align*}
{\| \F\bigl(T_{|\xi_3|m}[a,b_{\leq N},\an c_{>N}]\bigr) \|}_{L^\infty_\xi}
\lesssim {\|a \|}_{L_x^\infty} {\| b_{\leq N}\|}_{L_x^2} {\| c_{>N} \|}_{\dot H_x^{-1}}
\lesssim N^{-\frac12} {\| a \|}_{L_x^\infty} {\|\wh{b} \|}_{L_\xi^\infty}  {\| c \|}_{L_x^2}.
\end{align*}
Next, note that under our assumptions, $|\xi_2|m$ is also Coifman--Meyer.  Thus,
\begin{align*}
{\| \F\bigl(T_{|\xi_2|m}[a, \an b_{>N}, c_{>N}]\bigr) \|}_{L^\infty_\xi}
\lesssim {\|a \|}_{L_x^\infty} {\| b_{> N}\|}_{\dot H_x^{-1}} {\| c \|}_{L_x^2}
\lesssim N^{-\frac12}{\| a \|}_{L_x^\infty} {\|\wh{b} \|}_{L_\xi^\infty} {\| c \|}_{L_x^2}.
\end{align*}

{\it Proof of \eqref{TriEst3}.} Suppose $|\xi_3|m$ is a Coifman--Meyer symbol supported on a region where $|\xi_3|\gtrsim\max\{|\xi_1|,|\xi_2|\}$.

We can obtain the first estimate in \eqref{TriEst3} as follows:
\begin{align*}
{\| T_{|\xi_3|m}(a,b, \an c_{>N}) \|}_{L_x^2}
\lesssim {\|a\|}_{L_x^\infty} {\| b \|}_{L_x^\infty} {\|c_{>N} \|}_{\dot H_x^{-1}}
\lesssim  N^{-\frac12}{\| a \|}_{L_x^\infty} {\|b\|}_{L_x^\infty} {\| \wh{c} \|}_{L_\xi^\infty}.
\end{align*}

To obtain the second estimate in \eqref{TriEst3}, we proceed as above and decompose $b=b_{\leq N}+b_{>N}$.  Using \eqref{lem:bernstein} as well, we first have
\begin{align*}
{\| T_{|\xi_3|m}(a,b_{\leq N},\an c_{>N}) \|}_{L_x^2}
\lesssim {\|a\|}_{L_x^\infty} {\| b_{\leq N} \|}_{L_x^2} {\| \an c_{>N} \|}_{L_x^\infty}
\lesssim N^{-\frac12}{\| a \|}_{L_x^\infty} {\|\wh{b}\|}_{L_\xi^\infty} {\| c \|}_{L_x^\infty}.
\end{align*}
As $|\xi_2| m$ is Coifman--Meyer, we can also estimate
\begin{align*}
{\| T_{|\xi_2|m}(a, \an b_{>N}, c_{>N}) \|}_{L_x^2}
\lesssim {\|a\|}_{L_x^\infty} {\| b_{> N} \|}_{\dot H_x^{-1}} {\| c \|}_{L_x^\infty}
\lesssim N^{-\frac12}{\| a \|}_{L_x^\infty}  {\|\wh{b}\|}_{L_\xi^\infty} {\| c \|}_{L_x^\infty}.
\end{align*}
This completes the proof.
\end{proof}

		\section{Proof of Proposition~\ref{prop:fhat}}
		\label{sec:fhat}
		
	In this section we prove the estimate \eqref{eq:fhat bootstrap} for $\wh{f}(t).$ Using \eqref{eq:duhamel} we see that it suffices to estimate the following terms in $L_\xi^\infty$:
\begin{align}
\label{eq:fhat bootstrap1}
&\int_1^t\iint_{\R^2} e^{is\Phi}\, \who{f}(\xi_1)\who{f}(\xi_2)\who{f}(\xi_3)\ddd,\\
\label{eq:fhat bootstrap2}
&\int_1^t\iint_{\R^2} e^{is\Psi}\, \wh{f}(\xi_1)\wh{f}(\xi_2)\wh{f}(\xi_3)\ddd, \\
\label{eq:fhat bootstrap3}
&\int_1^t\iint_{\R^2} e^{is\Omega}\,\who{f}(\xi_1)\who{f}(\xi_2)\wh{f}(\xi_3)\ddd,\\
\label{eq:fhat bootstrap4}
&\int_1^t\F(e^{-is\partial_{xx}/2}\vert u\vert^2u\big)(s,\xi)\,ds,
\end{align}		
where the phases $\Phi,\Psi,\Omega$ are as in \eqref{Phases2} and we use the notation from \eqref{notfreq}.
	
		\subsection{Estimation of
			\texorpdfstring{\eqref{eq:fhat bootstrap1}}{Lg}}\label{sec:fhat1}

	We recall the notation from \eqref{eq:cutoff} and write
			$$1=
			[\phi_{lo}(\xi_1)+\phi_{hi}(\xi_1)]
			[\phi_{lo}(\xi_2)+\phi_{hi}(\xi_2)]
			[\phi_{lo}(\xi_3)+\phi_{hi}(\xi_3)]
			$$
in the integrand of \eqref{eq:fhat bootstrap1}. Expanding the product, we encounter two types of terms: (i) the low frequency term $\phi_{lo}(\xi_1)\phi_{lo}(\xi_2)\phi_{lo}(\xi_3)$, (ii) terms are of the form $\phi_*(\xi_j)\phi_*(\xi_k)\phi_{hi}(\xi_{\ell}),$ where $j,k,\ell\in\{1,2,3\}.$
				
	We estimate the contribution of the low frequency term by using volume bounds:
\begin{equation}
 \label{eq:fhat1 LLL}
\bigg\| \int_1^t\iint_{\R^2} e^{is\Phi} \who{f_{lo}}(\xi_1) \who{f_{lo}}(\xi_2) \who{f_{lo}}(\xi_3)\ddd \bigg\|_{L_\xi^\infty}
  \lesssim \int_1^t s^{-1}\|\wh{f}(s)\|_{L_\xi^\infty}^3\,ds,
\end{equation}
which is an acceptable contribution to the right-hand side of \eqref{eq:fhat bootstrap}.

For the terms of type (ii) we write $1=\chi_1(\vec{\xi})+\chi_2(\vec{\xi})+\chi_3(\vec{\xi}) $ for $\vec{\xi}\in\R^3$,
where each $\chi_j$ is a smooth Coifman--Meyer multiplier such that 
	\begin{equation}
	\label{eq:max specified0} 
	\vert\xi_j\vert\geq\max\{\tfrac{9}{10}\vert\xi_k\vert :k\neq j\}\quad\text{for all}\quad\vec{\xi}\in\text{support}(\chi_j).
	\end{equation}
See Appendix~\ref{section:appendix} for the construction of such multipliers.
We will show how to estimate the contribution from $\chi_3$. The same ideas suffice to treat the (almost symmetric) contributions from $\chi_1$ and $\chi_2$.  Note that on the support of $\chi_3$ we need only consider the contribution of the terms containing $\wh{f}_{hi}(\xi_3)$; indeed, if $|\xi_3| \lesssim s^{-\frac{1}{2}}$, then $\max_j\vert\xi_j\vert\lesssim s^{-\frac12}$ and we can estimate with volume bounds as we did for \eqref{eq:fhat1 LLL} above.  Thus, it suffices to consider the contribution of the term
\begin{equation}
\label{eq:fhat1 ii} \int_1^t\iint_{\R^2} e^{is\Phi} \chi_3(\vec{\xi}) \who{f_*}(\xi_1) \who{f_*}(\xi_2) \who{f_{hi}}(\xi_3)\ddd.
\end{equation}
In the region of integration in \eqref{eq:fhat1 ii} we have $\Phi\neq 0$, and in particular
$|\Phi| \gtrsim |\vec{\xi}|^2 \sim \xi_3^2$. We may therefore use the identity $e^{is\Phi} = (i\Phi)^{-1}\partial_s e^{is\Phi}$
and integrate by parts to write
		\begin{align}
		\eqref{eq:fhat1 ii}&=
			\label{eq:fhat1 p1}	
		\bigg[ \iint_{\R^2} \dfrac{e^{is\Phi}}{i\Phi} \chi_3(\vec{\xi}) \,
		\who{f_*}(\xi_1)\who{f_*}(\xi_2)\who{f_{hi}}(\xi_3)
		\dd\bigg]_{s=1}^t
		\\&\quad\label{eq:fhat1 p2}	
		-\int_1^t\iint_{\R^2} \dfrac{e^{is\Phi}}{i\Phi} \chi_3(\vec{\xi})
		\partial_s[\phi_*(\xi_1)\phi_*(\xi_2)\phi_{hi}(\xi_3)] \,
		\who{f}(\xi_1)\who{f}(\xi_2)\who{f}(\xi_3)\ddd
		\\&\quad\label{eq:fhat1 p3}	
		-\int_1^t\iint_{\R^2} \dfrac{e^{is\Phi}}{i\Phi} \chi_3(\vec{\xi})\phi_*(\xi_1)
		[\partial_s\who{f}(\xi_1)]
		\who{f_*}(\xi_2)\who{f_{hi}}(\xi_3)\ddd
		\\&\quad\label{eq:fhat1 p4}	
		-\int_1^t\iint_{\R^2}\dfrac{e^{is\Phi}}{i\Phi} \chi_3(\vec{\xi}) \,
		\who{f_*}(\xi_1)\phi_*(\xi_2)
		[\partial_s\who{f}(\xi_2)]\who{f_{hi}}(\xi_3)\ddd
		\\&\quad \label{eq:fhat1 p5} 
		-\int_1^t\iint_{\R^2}\dfrac{e^{is\Phi}}{i\Phi} \chi_3(\vec{\xi}) \,
		\who{f_*}(\xi_1)\who{f_*}(\xi_2)
		\phi_{hi}(\xi_3)[\partial_s\who{f}(\xi_3)]\ddd.
		\end{align}

Using the notation from \eqref{def:trilinear}, we notice that we can write
\begin{equation*}
\eqref{eq:fhat1 p1}= \Big[ e^{\frac{is\xi^2}{2}}\F\big(T_{m}[\bar{u}_*,\bar{u}_*, \bar{u}_{hi}]\big)(s,\xi) \Big]_{s=1}^t
               \end{equation*}
where $m = \chi_3(\vec{\xi})(i\Phi)^{-1}$ is a symbol satisfying the hypotheses of Lemma~\ref{lem:tri}(i).  That is, $m$ is supported on a region where $|\xi_3|\gtrsim\max\{|\xi_1|,|\xi_2|\}$, and one can check that $\xi_3^2m$ is Coifman--Meyer.  We apply \eqref{TriEst1} with $N\sim s^{-\frac12}$ (cf. \eqref{eq:cutoff}) to obtain
$$\|\eqref{eq:fhat1 p1}\|_{L_\xi^\infty}
	\lesssim \|\wh{f}(1)\|_{L_\xi^\infty}^2 \|u(1)\|_{L_x^\infty}
	+t^{\frac12}\|\wh{f}(t)\|_{L_\xi^\infty}^2 \|u(t)\|_{L_x^\infty}.$$
In view of Lemma \ref{lem:decay}, this is an acceptable contribution to the right-hand side of \eqref{eq:fhat bootstrap}.


	We next turn to \eqref{eq:fhat1 p2}.
From the definition of the cutoffs $\phi_{lo}$ and $\phi_{hi}$ (cf. \eqref{eq:cutoff}), we have $\partial_s\phi_*(\xi_j) = \pm (1/2) s^{-1/2}\xi_j \,\phi'(s^{1/2}\xi_j).$
As multiplication by $s^{1/2}\xi_j \phi'(s^{1/2}\xi_j)$ corresponds to a bounded projection to frequencies of size $\sim s^{-1/2}$, we can write $\partial_s\phi_*(\xi_j)\wh{f}(\xi_j) = s^{-1} \wh{f_{med}}(\xi_j)$, where $f_{med}$ denotes such a projection of $f$.  Distributing the derivatives and considering all of the possibilities, one can see that to treat \eqref{eq:fhat1 p2} it ultimately suffices to show how to estimate a term such as the following:
	\begin{align}
	&\int_1^t\iint_{\R^2}s^{-1}
	\dfrac{e^{is\Phi}}{\Phi} \chi_3(\vec{\xi}) \, \who{f_*}(\xi_1)\who{f_*}(\xi_2)\who{f_{med}}(\xi_3)\ddd.
	\label{eq:fhat1 p21}	
	\end{align}
This can be estimated as we did above for \eqref{eq:fhat1 p1}, using Lemma~\ref{lem:tri}(i):
	\begin{align*}
	\|\eqref{eq:fhat1 p21}\|_{L_\xi^\infty} & \lesssim
	\int_1^t s^{-\frac12} \, \|\wh{f}(s)\|_{L^\infty_\xi}^2\|u(s)\|_{L^\infty_x} \,ds,
	\end{align*}
which is acceptable in view of \eqref{eq:Xdecay}. 
	
	
	We next turn to \eqref{eq:fhat1 p3}. Noting that
\begin{equation}
\label{dsf}
e^{is\partial_{xx}/2}\partial_s f = (\partial_s +\tfrac{i}{2}\partial_{xx})u = \text{\O}(u^3),
\end{equation}
we can  use  Lemma~\ref{lem:tri}(i) with $(a,b,c) = (u_*,e^{is\partial_{xx}/2}\partial_s f_*, u_{hi})$
to get:
		\begin{align*}
		\|\eqref{eq:fhat1 p3}\|_{L_\xi^\infty}
		& \lesssim\int_1^t
		s^{\frac{1}{2}} \| \wh{u}(s) \|_{L_\xi^\infty} \|\text{\O}(u^3(s))\|_{L_x^\infty} \|\wh{u}(s)\|_{L_\xi^\infty}\,ds
		\lesssim \int_1^t s^{-1}[s^{\frac12}\|u(s)\|_{L_x^\infty}]^3 \|\wh{f}(s)\|_{L_\xi^\infty}^2\,ds,
		\end{align*}
which is acceptable.  We can treat \eqref{eq:fhat1 p4} in the same way, as we can exchange the role of $a$ and $b$ in \eqref{TriEst1}.


The term \eqref{eq:fhat1 p5} can be treated similarly, using
the second inequality in \eqref{TriEst1} with $(a,b,c)=(u_*,u_*,P_{hi}e^{is\partial_{xx}/2}\partial_s f)$:
\begin{align*}
{\| \eqref{eq:fhat1 p5} \|}_{L_\xi^\infty} & \lesssim \int_1^t s^{\frac{1}{2}} \| \wh{f}(s) \|_{L_\xi^\infty}^2 \|\text{\O}(u^3(s))\|_{L_x^\infty} \, ds,
\end{align*}
which is acceptable (cf. \eqref{eq:Xdecay}).  This completes the estimation 
\eqref{eq:fhat bootstrap1}. 

		\subsection{Estimation of \texorpdfstring{\eqref{eq:fhat bootstrap2}}{Lg}} \label{sec:fhat2}
		
As before we write
		$$1=
		[\phi_{lo}(\xi_1)+\phi_{hi}(\xi_1)]
		[\phi_{lo}(\xi_2)+\phi_{hi}(\xi_2)]
		[\phi_{lo}(\xi_3)+\phi_{hi}(\xi_3)]
		$$
in the integrand of \eqref{eq:fhat bootstrap2}.  We estimate the contribution of $\phi_{lo}(\xi_1)\phi_{lo}(\xi_2)\phi_{lo}(\xi_3)$ by volume bounds, as in \eqref{eq:fhat1 LLL}.

For the remaining terms we once again write $1=\chi_1(\vec{\xi}) +\chi_2(\vec{\xi}) +\chi_3(\vec{\xi}) $ for $\vec{\xi}\in\R^3$,
where each $\chi_j$ is a smooth Coifman--Meyer multiplier such that \eqref{eq:max specified0} holds.  We will show how to estimate the contribution of $\chi_2$. Similar ideas suffice to treat the contribution of $\chi_1$ and $\chi_3$
(see Remark \ref{remark:fhat} below for more details).  We need only consider the contribution of $\chi_2$ in terms containing $\wh{f}_{hi}(\xi_2)$,
since if $|\xi_2| \lesssim s^{-\frac{1}{2}}$, then $\max_j\vert\xi_j\vert\lesssim s^{-\frac12}$
and we can simply estimate using volume bounds, as we did for the low frequency term.

	On the support of $\chi_2(\vec{\xi})$ we further decompose
	$1= \chi_\eta(\vec{\xi})+ \chi_\sigma(\vec{\xi})+ \chi_s(\vec{\xi}),$
	and let $\chi_{2,*}:=\chi_2\chi_*$ be smooth Coifman--Meyer multipliers such that
	\begin{align}
	&\vert\xi_1-\xi_2\vert\geq\tfrac{1}{100}\vert\xi_2\vert
	\quad\text{for}\quad
	\vec{\xi}\in\text{support}(\chi_\eta),
	\label{eq:chi-eta}	
	\\
	&\vert\xi_3-\xi_2\vert\geq\tfrac{1}{100}\vert\xi_2\vert
	\quad\text{for}\quad
	\vec{\xi}\in\text{support}(\chi_\sigma),
	\label{eq:chi-sigma}	
	\\
	&\vert\xi_1-\xi_2\vert\leq\tfrac{1}{50}\vert\xi_2\vert
	\quad\text{and}\quad
	\vert\xi_3-\xi_2\vert\leq\tfrac{1}{50}\vert\xi_2\vert
	\quad\text{for}\quad
	\vec{\xi}\in\text{support}(\chi_s).
	\label{eq:chi-s}		
	\end{align}
See Appendix~\ref{section:appendix} for the construction of such multipliers.
The subscripts indicate the variable with respect to which we will integrate by parts.
According to this decomposition of the frequency space, we are faced with estimating the following three terms:
	\begin{align}
	&\int_1^t \iint_{\R^2} e^{is\Psi}
	\chi_{2,\eta}(\vec{\xi})
	\wh{f_*}(\xi_1)\wh{f_{hi}}(\xi_2)\wh{f_*}(\xi_3) \ddd,
	\label{eq:fhat2 eta} 
	\\
	&\int_1^t\iint_{\R^2} e^{is\Psi}
	\chi_{2,\sigma}(\vec{\xi})
	\wh{f_*}(\xi_1)\wh{f_{hi}}(\xi_2)\wh{f_*}(\xi_3)\ddd,
	\label{eq:fhat2 sig} 
	\\
	&\int_1^t \iint_{\R^2} e^{is\Psi}
	\chi_{2,s}(\vec{\xi})
	\wh{f_*}(\xi_1)\wh{f_{hi}}(\xi_2)\wh{f_*}(\xi_3)\ddd.
	\label{eq:fhat2 s}
	\end{align}


\subsubsection{Estimation of \eqref{eq:fhat2 eta}}
Using \eqref{Phases2}--\eqref{Phases3}, we see that on the support of $\chi_{2,\eta}$ we have
	\begin{equation}
	\label{eq:fhat2 eta lb}    
	\vert\partial_\eta \Psi \vert
	=\vert \xi_2-\xi_1\vert
	\gtrsim \vert \xi_2\vert
	\gtrsim \vert \vec{\xi}\vert.
	\end{equation}
Thus we can use the identity $e^{is\Psi} = \partial_\eta e^{is\Psi} (is\partial_\eta\Psi)^{-1}$ and integrate by parts in $\eta$ to write
	\begin{align}
	\eqref{eq:fhat2 eta}
	&=-\int_1^t\iint_{\R^2} e^{is\Psi}
	\partial_\eta\big(\dfrac{1}{is\partial_\eta\Psi}\big)
	\chi_{2,\eta}(\vec{\xi})
	\wh{f_*}(\xi_1)\wh{f_{hi}}(\xi_2)\wh{f_*}(\xi_3)\ddd
	\label{eq:fhat2 eta1} 		
	\\ &\quad-\int_1^t\iint_{\R^2} e^{is\Psi}
	\dfrac{\partial_\eta[\chi_{2,\eta}(\vec{\xi})
	\phi_*(\xi_1)\phi_{hi}(\xi_2)]}
	{is\partial_\eta\Psi}
	\wh{f}(\xi_1)\wh{f}(\xi_2)\wh{f_*}(\xi_3)\ddd
	\label{eq:fhat2 eta2}		
	\\ &\quad-\int_1^t\iint_{\R^2} \dfrac{e^{is\Psi}}
	{is\partial_\eta\Psi}
	\chi_{2,\eta}(\vec{\xi})
	\phi_*(\xi_1)\partial_\eta\wh{f}(\xi_1)\wh{f_{hi}}(\xi_2)\wh{f_*}(\xi_3)\ddd
	\label{eq:fhat2 eta3}		
	\\ &\quad-\int_1^t\iint_{\R^2} \dfrac{e^{is\Psi}}
	{is\partial_\eta\Psi}
	\chi_{2,\eta}(\vec{\xi})
	\wh{f_*}(\xi_1)\phi_{hi}(\xi_2)\partial_\eta\wh{f}(\xi_2)\wh{f_*}(\xi_3)\ddd.
	\label{eq:fhat2 eta4}		
	\end{align}
	
	We first estimate \eqref{eq:fhat2 eta1}. Using \eqref{Phases3}, we see that $\partial_\eta(1/\partial_\eta\Psi) = -2|\xi_2-\xi_1|^{-2}$.  Recalling \eqref{eq:fhat2 eta lb} and the fact that $|\xi_2|\gtrsim\max\{|\xi_1|,|\xi_2|\}$ in the integral above, we can write
\begin{align*}
\eqref{eq:fhat2 eta1} = \int_{1}^t (is)^{-1} e^{\frac{is|\xi|^2}{2}} \mathcal{F} \big( T_{m}[u_*, u_{hi}, u_*]\big)(s,\xi)  \, ds,
\end{align*}
where the symbol $m = \partial_\eta\big(1/\partial_\eta\Psi\big) \chi_{2,\eta}$
satisfies the hypotheses of Lemma~\ref{lem:tri}(i).  Here and throughout Section~\ref{sec:fhat2}, $\xi_2$ plays the role of $\xi_3$ in the application of Lemma~\ref{lem:tri}.  Applying \eqref{TriEst1}, we obtain an acceptable contribution:
\begin{align*}
\| \eqref{eq:fhat2 eta1} \|_{L_\xi^\infty} & \lesssim \int_{1}^t  s^{-\frac12} \, \| \wh{f}(s) \|_{L_\xi^\infty}^2 \|u(s)\|_{L_x^\infty} \, ds.
\end{align*}


We next turn to \eqref{eq:fhat2 eta2}. Two types of terms arise, depending on where $\partial_\eta$ lands.
First, if $\partial_\eta$ lands on $\chi_{2,\eta}$ we are led to consider the following:
\begin{align}
& \int_1^t\iint_{\R^2} s^{-1}
\, e^{is\Psi} \dfrac{\partial_\eta\chi_{2,\eta}(\vec{\xi})}{\partial_\eta\Psi} \wh{f_*}(\xi_1)\wh{f_{hi}}(\xi_2)\wh{f_*}(\xi_3)\ddd.
\label{eq:fhat2 eta21}	
\end{align}
Second, we note that $\partial_\eta\phi_*(\xi_j) = \pm s^{\frac12}\phi'(s^{\frac12}\xi_j)$,
and that multiplication by $\phi'(s^{\frac12}\cdot)$ corresponds to a projection to frequencies $\sim s^{-\frac12}$.  As before, we denote this by $P_{med}f = f_{med}$.  Considering all of the possibilities, one can see that to treat the terms that arise when $\partial_\eta$
lands on one of the $\phi_*(\xi_j)$, it suffices to estimate the terms
	\begin{align}
\label{eq:fhat2 eta22}	
& \int_1^t\iint_{\R^2} s^{-\frac12} \, e^{is\Psi} \dfrac{\chi_{2,\eta}(\vec{\xi})}{\partial_\eta\Psi}
	\wh{f_*}(\xi_1) \wh{f_{med}}(\xi_2) \wh{f_*}(\xi_3)\ddd,
\\
\label{eq:fhat2 eta23}
& \int_1^t\iint_{\R^2} s^{-\frac12} \, e^{is\Psi} \dfrac{\chi_{2,\eta}(\vec{\xi})}{\partial_\eta\Psi}
	\wh{f_{med}}(\xi_1) \wh{f_{hi}}(\xi_2) \wh{f_*}(\xi_3)\ddd.
	\end{align}
Thus, to treat \eqref{eq:fhat2 eta2} it suffices to estimate \eqref{eq:fhat2 eta21}--\eqref{eq:fhat2 eta23}.

For \eqref{eq:fhat2 eta21} we use \eqref{eq:fhat2 eta lb} and the fact that $\xi_2\partial_\eta\chi_{2,\eta}(\vec{\xi})$ is Coifman--Meyer to write
\begin{align*}
\eqref{eq:fhat2 eta21} = \int_1^t s^{-1} e^{\frac{is\xi^2}{2}} \mathcal{F}\big(T_m[u_*, u_{hi},u]\big)(s,\xi)\,ds,
\end{align*}
where $m$ satisfies the hypotheses of Lemma~\ref{lem:tri}(i).  The estimate \eqref{TriEst1} then gives
\begin{align*}
\| \eqref{eq:fhat2 eta21} \|_{L_\xi^\infty} \lesssim \int_1^t s^{-\frac12} \|\wh{f}(s)\|_{L_\xi^\infty}^2 \|u(s)\|_{L_x^\infty} \,ds,
\end{align*}
which is acceptable.

Next, recalling once again \eqref{eq:fhat2 eta lb} and \eqref{eq:chi-eta}, we can write
\begin{equation*}
\eqref{eq:fhat2 eta22} = \int_1^t s^{-\frac12}  e^{\frac{is\xi^2}{2}} \F\big(T_m[u_*,u_{med},u_*]\big)(s,\xi)\,ds,
\end{equation*}
where $m$ is a symbol that satisfies the hypotheses of Lemma~\ref{lem:tri}(ii); that is, $|\xi_2| m$ is a Coifman--Meyer symbol supported on a region where $\xi_2$ is the largest frequency (up to a constant).  We apply \eqref{TriEst2} and \eqref{est0} to get the following acceptable estimate:
\begin{align*}
\| \eqref{eq:fhat2 eta22} \|_{L_\xi^\infty} \lesssim \int_1^t s^{-\frac14} \|u(s)\|_{L_x^\infty} \| u_{med} \|_{L_x^2} \|\wh{f}(s)\|_{L_\xi^\infty} \,ds
  \lesssim \int_1^t s^{-\frac12} \|u(s)\|_{L_x^\infty} \|\wh{f}(s)\|_{L_\xi^\infty}^2 \,ds.
 \end{align*}
As the term \eqref{eq:fhat2 eta23} can be estimated in the same way, this completes the treatment of \eqref{eq:fhat2 eta2}.


To estimate the term \eqref{eq:fhat2 eta3} we proceed similarly. We can write
\begin{equation*}
\eqref{eq:fhat2 eta3} = \int_1^t s^{-1}  e^{\frac{is\xi^2}{2}} \F\big(T_m[ P_*Ju, u_{hi},u_*]\big)(s,\xi)\,ds,
\end{equation*}
where $m$ is a symbol satisfying the hypotheses of Lemma~\ref{lem:tri}(ii). Using \eqref{TriEst2}, we obtain:
\begin{align*}
\| \eqref{eq:fhat2 eta3} \|_{L_\xi^\infty} \lesssim \int_1^t s^{-\frac34} \|u(s)\|_{L_x^\infty} \| Ju(s) \|_{L_x^2} \|\wh{f}(s)\|_{L_\xi^\infty} \,ds,
 \end{align*}
which is an acceptable contribution to the right-hand side of \eqref{eq:fhat bootstrap}.  The last term \eqref{eq:fhat2 eta4} can be estimated in the same way.



\subsubsection{Estimation of \eqref{eq:fhat2 sig}} This term is very similar to \eqref{eq:fhat2 eta}.
In particular, we note that on the support of $\chi_{2,\sigma}$ we have $\vert\partial_\sigma \Psi \vert
	=\vert \xi_2-\xi_3\vert
	\gtrsim \vert \xi_2\vert
	\gtrsim \vert \vec{\xi}\vert.$
Thus we can use the identity $e^{is\Psi} = (is\partial_\sigma\Psi)^{-1} \partial_\sigma e^{is\Psi}$
and integrate by parts in $\sigma$. The ideas used to estimate \eqref{eq:fhat2 eta} then suffice to handle the resulting terms.	
	

\subsubsection{Estimation of \eqref{eq:fhat2 s}} Using \eqref{Phases2} and \eqref{eq:chi-s},
we note that on the support of $\chi_{2,s}$, we have
	\begin{align}
	\vert\Psi \vert
	& =\vert (\xi_1-\xi_2)\xi_2+\xi_2(\xi_3-\xi_2)
		+\xi_1\xi_3+2\xi_2^2 \vert
	\nonumber
	\\
	& \geq 2\vert \xi_2\vert^2-\vert\xi_2\vert
		(\vert\xi_1-\xi_2\vert +
		\vert\xi_3-\xi_2\vert)
		-\vert\xi_1\vert \vert\xi_3\vert
	\geq \vert\xi_2\vert^2 \sim \vert \xi_1\vert^2\sim\vert\xi_3\vert^2 \sim \vert\vec{\xi}\vert^2.	
	\label{eq:fhat2 s lb}    
	\end{align}
We integrate by parts in $s$ using the identity $e^{is\Psi} = (i\Psi)^{-1}\partial_s e^{is\Psi}$. This yields
	\begin{align}
	\eqref{eq:fhat2 s}&=
	\bigg[ \iint_{\R^2} \dfrac{e^{is\Psi}\chi_{2,s}(\vec{\xi})}
	{i\Psi}\wh{f_*}(\xi_1)\wh{f_{hi}}(\xi_2)\wh{f_*}(\xi_3)
	\dd\bigg]_{s=1}^t
	\label{eq:fhat2 s1}	
	\\
	&\quad-\int_1^t\!\!\!\iint_{\R^2}\!\!
	\dfrac{e^{is\Psi}\chi_{2,s}(\vec{\xi})\partial_s[\phi_*(\xi_1)\phi_{hi}(\xi_2)\phi_*(\xi_3)]}{i\Psi}
	\wh{f}(\xi_1)\wh{f}(\xi_2)\wh{f}(\xi_3)\ddd
	\label{eq:fhat2 s2}	
	\\
	& \quad - \int_1^t\iint_{\R^2}
	\dfrac{e^{is\Psi} \chi_{2,s}(\vec{\xi}) \phi_*(\xi_1) \phi_{hi}(\xi_2) \phi_*(\xi_3)}{i\Psi}
	\partial_s\bigl[ \wh{f}(\xi_1) \wh{f}(\xi_2) \wh{f}(\xi_3) \bigr] \ddd.
	\label{eq:fhat2 s3} 	
	\end{align}
In light of the lower bound \eqref{eq:fhat2 s lb}, these terms are similar to those in \eqref{eq:fhat1 p1}--\eqref{eq:fhat1 p5}.

First, using \eqref{eq:fhat2 s lb}, we can write \eqref{eq:fhat2 s1} as
\begin{equation*}
\Big[ e^{\frac{is|\xi|^2}{2}} \mathcal{F}\big( T_m [u_*,u_{hi},u_*] \big)(s,\xi) \Big]_{s=1}^t,
\end{equation*}
where the symbol $m = \chi_2(i\Psi)^{-1}$ satisfies the hypotheses of Lemma~\ref{lem:tri}(i) (up to exchanging the role of $\xi_3$ and $\xi_2$, as above).  The estimate \eqref{TriEst1} (applied, as always, with $N\sim s^{-\frac12}$) yields
\begin{align*}
\| \eqref{eq:fhat2 s1}  \|_{L_\xi^\infty} \lesssim \|u(1)\|_{L_x^\infty} \|\wh{f}(1)\|^2_{L_\xi^\infty}  + t^{\frac12} \, \|u(t)\|_{L_x^\infty} \|\wh{f}(t)\|^2_{L_\xi^\infty},
\end{align*}
which is an acceptable contribution to the right-hand side of \eqref{eq:fhat bootstrap}.


We next turn to \eqref{eq:fhat2 s2}.
As observed earlier, we can write $\partial_s\phi_*(\xi_j)\wh{f}(\xi_j) = s^{-1} \wh{f_{med}}(\xi_j)$, where $f_{med}$ denotes
the projection of $f$ to frequencies $\sim s^{-1/2}$.
Considering all of the possibilities, one can see that to treat \eqref{eq:fhat2 s2} it ultimately suffices to show how to bound a term such as
	\begin{equation}
	\label{eq:fhat2 s21}
	\int_1^t\iint_{\R^2} s^{-1} \dfrac{e^{is\Psi}\,\chi_{2,s}(\vec{\xi})}{\Psi}
	\wh{f_*}(\xi_1) \wh{f_{med}}(\xi_2) \wh{f_*}(\xi_3)\ddd.
	\end{equation}
We can estimate this term as we did \eqref{eq:fhat1 p21}, using \eqref{TriEst1}.
	
For the term \eqref{eq:fhat2 s3}, we can proceed as we did for \eqref{eq:fhat1 p3}--\eqref{eq:fhat1 p5}: the hypotheses of Lemma~\ref{lem:tri}(i) hold, and we can use \eqref{TriEst1} and \eqref{dsf} to estimate
\begin{align*}
\| \eqref{eq:fhat2 s3}  \|_{L_\xi^\infty} \lesssim \int_1^t s^{\frac{1}{2}} \| \wh{f} \|_{L_\xi^\infty}^2 \|\text{\O}(u^3(s))\|_{L_x^\infty} \, ds,
\end{align*}
which is acceptable.


\begin{remark}\label{remark:fhat}
We have now estimated \eqref{eq:fhat2 eta}--\eqref{eq:fhat2 s}.  This handles the contribution to \eqref{eq:fhat bootstrap2}
associated with the cutoff $\chi_2$, that is, the region where $|\xi_2| \gtrsim \max\{|\xi_1|,|\xi_3|\}$ (cf. \eqref{eq:max specified0}). We now briefly discuss how to treat the terms containing $\chi_1$ or $\chi_3$.

In estimating the contribution from $\chi_2$, the key idea was to decompose frequency space into regions
such that at least one of $\vert\partial_\eta\Psi\vert$, $\vert\partial_\sigma\Psi\vert$, or $\vert\Psi\vert$ was suitably bounded below.
In the support of $\chi_1$, using \eqref{Phases3}, we can achieve such a decomposition as follows:
\begin{itemize}	
\item First, if $\vert\xi_1-\xi_2\vert\geq\tfrac{1}{100}\vert\xi_1\vert$
then $\vert\partial_\eta\Psi\vert\gtrsim \vert\xi_1\vert\gtrsim\vert\vec{\xi}\vert$.

\item Next, if $\vert\xi_1-\xi_2\vert\leq\tfrac{1}{50}\vert\xi_1\vert$ and $\vert\xi_3-\xi_2\vert\geq\tfrac{1}{100}\vert\xi_2\vert$,
then $\vert\partial_\sigma\Psi\vert\gtrsim\vert\xi_2\vert\gtrsim\vert\vec{\xi}\vert$.

\item Finally, if $\vert\xi_1-\xi_2\vert\leq\tfrac{1}{50}\vert\xi_1\vert$ and $\vert\xi_3-\xi_2\vert\leq\tfrac{1}{50}\vert\xi_2\vert$
then $\vert\Psi\vert\gtrsim\vert\xi_1\vert^2 \gtrsim \vert\vec{\xi}\vert^2$.
\end{itemize}
Thus, we can use arguments similar to the ones above to handle the contribution of $\chi_1$. Similar ideas also suffice to treat the contribution of $\chi_3.$
\end{remark}


\subsection{Estimation of \texorpdfstring{\eqref{eq:fhat bootstrap3}}{Lg}}\label{sec:fhat bootstrap3}
We can estimate \eqref{eq:fhat bootstrap3} in a very similar manner to \eqref{eq:fhat bootstrap2}.
To wit, we split each function into low and high frequency pieces, and we handle the term containing all low frequencies with volume bounds.
For the remaining terms, we decompose frequency space into regions where one of $\vert\xi_1\vert,\vert\xi_2\vert,\vert\xi_3\vert$ is (almost) the maximum,
according to \eqref{eq:max specified0}.
On each such region, we decompose into regions where we have suitable lower bounds on either the phase $\Omega$ or its derivatives.
Consider for example the region where $\vert\xi_2\vert\geq\max\{\tfrac{9}{10}\vert\xi_1\vert, \tfrac{9}{10}\vert\xi_3\vert\}.$
Then we can define cutoffs $\chi_\eta,\chi_\sigma,$ and $\chi_s$ so that $1=\chi_\eta+\chi_\sigma+\chi_s$ and
\begin{align*}
& \vert\xi_2-\xi_1\vert\geq\tfrac{1}{100}\vert\xi_2\vert \quad \text{for}\quad \vec{\xi}\in\text{support}(\chi_\eta),
\\
& \vert\xi_2+\xi_3\vert\geq\tfrac{1}{100}\vert\xi_2\vert \quad \text{for}\quad \vec{\xi}\in\text{support}(\chi_\sigma),
\\
&\vert\xi_2-\xi_1\vert\leq\tfrac{1}{50}\vert\xi_2\vert
  \quad \text{and}\quad \vert\xi_2+\xi_3\vert\leq\tfrac{1}{50}\vert\xi_2\vert
  \quad \text{for}\quad \vec{\xi}\in\text{support}(\chi_s).
\end{align*}
From the formulas \eqref{Phases2}--\eqref{Phases3} it is easy to see that we have suitable lower bounds for $\partial_\eta\Omega$
and $\partial_\sigma\Omega$ in the support of $\chi_\eta$ and $\chi_\sigma$, respectively.
Furthermore, we claim that $\vert\Omega\vert\gtrsim\vert\xi_2\vert^2\gtrsim \vert\vec{\xi}\vert^2$
for $\vec{\xi}$ in the support of $\chi_s$.
Indeed, we can write
\begin{align*}
\Omega=\xi_2^2+\xi_1(\xi_1-\xi_2)+\xi_1(\xi_2+\xi_3) +\xi_2(\xi_1-\xi_2)+\xi_2(\xi_2+\xi_3)
\end{align*}
and note that in the support of $\chi_s$, we have
$$\vert \xi_1(\xi_1-\xi_2)+\xi_1(\xi_2+\xi_3) +\xi_2(\xi_1-\xi_2)+\xi_2(\xi_2+\xi_3)\vert \leq \tfrac{3}{50}\xi_2^2.$$
Thus, proceeding as in the case of \eqref{eq:fhat bootstrap2}, we can deal with the term \eqref{eq:fhat bootstrap3}.
As the analysis is quite similar, we omit the details.

\subsection{Estimation of \texorpdfstring{\eqref{eq:fhat bootstrap4}}{Lg}}
We can handle the term \eqref{eq:fhat bootstrap4} in a relatively simple manner due to the gauge-invariance of $\vert u\vert^2 u$.
Using \eqref{eq:factorization} we can first rewrite
$$\eqref{eq:fhat bootstrap4}=\int_1^t s^{-1}\F \bar{M}(s)\F^{-1}\big(\vert \F \M f\vert^2\F \M f\big)(s,\xi)\,ds.$$
Writing $\F \b{M}(s)\F^{-1}=1+\F[\b{M}(s)-1]\F^{-1},$ it suffices to estimate the following:
	\begin{align}
	\label{eq:gauge1}
	&\int_1^t s^{-1}\big(\vert \F\M f\vert^2\F\M f\big)(s,\xi)\,ds,
	\\
	\label{eq:gauge2}
	&\int_1^t s^{-1}\F[\b{M}(s)-1]\F^{-1}\big(\vert\F\M f\vert^2
	\F\M f\big)(s,\xi)\,ds.
	\end{align}
	
Arguing as in the proof of Lemma~\ref{lem:decay} (see Section~\ref{sec:decay}), we find
	\begin{equation}
	\label{gauge X}
	\|\F\M(s) f(s)\|_{L_\xi^\infty}\lesssim \|u(s)\|_{X(s)},
	\end{equation}
and hence $\|\eqref{eq:gauge1}\|_{L_\xi^\infty}
	\lesssim \int_1^t s^{-1}\|u(s)\|_X^3\,ds,
	$
which is acceptable.

For \eqref{eq:gauge2} we again argue as in the proof of Lemma~\ref{lem:decay}. We split at frequency $\sqrt{s}$, using the operators $\bar{P}_{\leq N}=\F\phi(\cdot/N)\F^{-1}.$  For the high frequencies, we use \eqref{lem:bernstein}, Plancherel, the chain rule, and \eqref{eq:forms of Ju} to estimate
	\begin{align*}
	\|\bar{P}_{>\sqrt{s}}\F [\b{M}(s)-1]\F^{-1}\big(
	\vert \F \M f\vert^2\F\M f\big)(s)\|_{L_\xi^\infty} &\lesssim s^{-\frac14}\|\partial_x\F[\b{M}(s)-1]\F^{-1}\big(
	\vert \F \M f\vert^2\F\M f\big)(s)\|_{L_x^2}
	\\ &\lesssim s^{-\frac14}\|\partial_x\big(
	\vert \F \M f\vert^2\F\M f\big)(s)\|_{L_x^2}
	\\ &\lesssim s^{-\frac14}\|\F\M f(s)\|_{L_\xi^\infty}^2
			\|\partial_x\F \M f(s)\|_{L_x^2}
	\lesssim \|u(s)\|_X^3.
	\end{align*}
	
For the low frequencies we use 
the pointwise bound \eqref{eq:pw}, Plancherel, 
and \eqref{eq:forms of Ju} to estimate
	\begin{align*}
	\|\bar{P}_{\leq\sqrt{s}}\F [\b{M}(s)-1]\F^{-1}\big(
	\vert \F \M f\vert^2\F\M f\big)(s)\|_{L_\xi^\infty} &\lesssim \|\phi(\tfrac{\cdot}{\sqrt{s}})[\b{M}(s)-1]\F^{-1}
	\big(\vert \F \M f\vert^2\F\M f\big)(s)\|_{L_x^1}
	\\ &\lesssim s^{\frac14}\|[\b{M}(s)-1]\F^{-1}
	\big(\vert \F \M f\vert^2\F\M f\big)(s)\|_{L_x^2}
	\\ &\lesssim s^{-\frac14}\|x \F^{-1}
	\big(\vert \F \M f\vert^2\F\M f\big)(s)\|_{L_x^2}
	\\ &\lesssim s^{-\frac14}\|\F\M f\|_{L_\xi^\infty}^2\|\partial_x\F Mf(s)\|_{L_x^2}
	\lesssim \|u(s)\|_X^3.
	\end{align*}
Thus $\|\eqref{eq:gauge2}\|_{L_\xi^\infty} \lesssim \int_1^t s^{-1}\|u(s)\|_X^3\,ds,$ which is acceptable.
This completes the estimation of \eqref{eq:fhat bootstrap4}, which in turn completes the proof of Proposition~\ref{prop:fhat}.

		\section{Proof of Proposition~\ref{prop:Ju}}
		\label{sec:Ju}
In this section we prove the estimate \eqref{eq:Ju bootstrap} for $Ju(t)$.
Equivalently we will estimate $\partial_\xi\wh{f}(t)$ in $L_\xi^2$, cf. \eqref{eq:forms of Ju}. Using \eqref{eq:duhamel} and recalling the notation from \eqref{notfreq} and \eqref{Phases2}, we can write
	\begin{align}
	\partial_\xi\wh{f}(t)= \partial_\xi\wh{f}(1)
	&-i\lambda_1(2\pi)^{-1}\int_1^t\iint_{\R^2} e^{is\Phi}[\partial_\xi
		\who{f}(\xi_1)]\who{f}(\xi_2)\who{f}(\xi_3)\ddd
	\label{eq:Ju easy1}
	\\ &\quad-i\lambda_2(2\pi)^{-1}\int_1^t\iint_{\R^2} e^{is\Psi}[\partial_\xi
		\wh{f}(\xi_1)]\wh{f}(\xi_2)\wh{f}(\xi_3)\ddd
	\label{eq:Ju easy2}
	\\ &
	\quad-i\lambda_3(2\pi)^{-1}\int_1^t\iint_{\R^2}e^{is\Omega}[\partial_\xi\who{f}(\xi_1)]
	\who{f}(\xi_2)\wh{f}(\xi_3)\ddd
	\label{eq:Ju easy3}
	\\ &\quad-i\lambda_1(2\pi)^{-1}\int_1^t\iint_{\R^2}e^{is\Phi}[is\partial_\xi\Phi]
		\who{f}(\xi_1)\who{f}(\xi_2)\who{f}(\xi_3)\ddd
	\label{eq:Ju1}
	\\ &\quad-i\lambda_2(2\pi)^{-1}\int_1^t\iint_{\R^2}e^{is\Psi}[is\partial_\xi\Psi]
		\wh{f}(\xi_1)\wh{f}(\xi_2)\wh{f}(\xi_3)\ddd
	\label{eq:Ju2}
	\\ &\quad-i\lambda_3(2\pi)^{-1}\int_1^t\iint_{\R^2}e^{is\Omega}[is\partial_\xi\Omega]
	\who{f}(\xi_1)\who{f}(\xi_2)\wh{f}(\xi_3)\ddd
	\label{eq:Ju3}
	\\ &\quad-i\lambda_4\int_1^t\partial_\xi\F\big(e^{-is\partial_{xx}/2}\vert u\vert^2 u)(s,\xi)\,ds
	\label{eq:Ju gauge}.
	\end{align}
	
	For the terms \eqref{eq:Ju easy1}, \eqref{eq:Ju easy2}, and \eqref{eq:Ju easy3} we recall that $e^{is\partial_{xx}/2}xf=Ju(s)$. Thus
	\begin{align*}
	\|\eqref{eq:Ju easy1}
	&+\eqref{eq:Ju easy2}
	+ \eqref{eq:Ju easy3}\|_{L_\xi^2} \lesssim \int_1^t \| Ju(s) \, u^2(s)\|_{L_x^2}\,ds
	\lesssim \int_1^t \|u(s)\|_{L_x^\infty}^2 \|Ju(s)\|_{L_x^2}\,ds,
	\end{align*}
which (in light of Lemma~\ref{lem:decay}) is bounded by the right-hand side of \eqref{eq:Ju bootstrap}.
It remains to estimate \eqref{eq:Ju1} through \eqref{eq:Ju gauge}.


\subsection{Estimation of \texorpdfstring{\eqref{eq:Ju1}}{Lg}} \label{sec:Ju1}
We recall the notation from \eqref{eq:cutoff} and write
	$$
	1=
	[\phi_{lo}(\xi_1)+\phi_{hi}(\xi_1)]
	[\phi_{lo}(\xi_2)+\phi_{hi}(\xi_2)]
	[\phi_{lo}(\xi_3)+\phi_{hi}(\xi_3)]
	$$
in the integrand of \eqref{eq:Ju1}. We expand the product and encounter two types of terms: (i) $\phi_{lo}(\xi_1)
		\phi_{lo}(\xi_2)\phi_{lo}(\xi_3)$, (ii) $\phi_*(\xi_j)\phi_*(\xi_k)\phi_{hi}(\xi_\ell),$
where $j,k,\ell\in\{1,2,3\}$.

	We estimate the contribution of term (i) by volume bounds:
	\begin{equation}\label{eq:Ju10}
	\bigg\|\int_1^t \iint_{\R^2} e^{is\Phi}
	[s\partial_\xi\Phi]\who{f_{lo}}(\xi_1)
	\who{f_{lo}}(\xi_2)\who{f_{lo}}(\xi_3)\ddd
	\bigg\|_{L_\xi^2}
	\lesssim \int_1^t s^{-\frac34}
	\|\wh{f}(s)\|_{L_\xi^\infty}^3\,ds,
	\end{equation}
which is acceptable.

We now turn to the terms of type (ii).
As in Section~\ref{sec:fhat}, we write $1=\chi_1(\vec{\xi})+\chi_2(\vec{\xi})+\chi_3(\vec{\xi}) $ for $\vec{\xi}\in\R^3$ so that \eqref{eq:max specified0} holds,
and note that it suffices to show how to estimate
	\begin{equation}
	\label{eq:Ju11}	
	\int_1^t \iint_{\R^2}
	e^{is\Phi}[is\partial_\xi\Phi] \chi_3(\vec{\xi}) \, \who{f_*}(\xi_1)\who{f_*}(\xi_2) \who{f_{hi}}(\xi_3)\ddd.
	\end{equation}

In the region of integration in \eqref{eq:Ju11} we have $\Phi\neq 0$; in fact, $|\Phi| \gtrsim |\vec{\xi}|^2$.
We may therefore use the identity $ e^{is\Phi}= (i\Phi)^{-1} \partial_s e^{is\Phi}$ and integrate by parts to write
	\begin{align}
	\eqref{eq:Ju11}&
	= \bigg[ \iint_{\R^2}e^{is\Phi} \dfrac{s\partial_\xi\Phi}{\Phi} \chi_3(\vec{\xi}) \, \who{f_*}(\xi_1)\who{f_*}(\xi_2)\who{f_{hi}}(\xi_3)\dd\bigg]_{s=1}^t
	\label{eq:Ju111}	
	\\
& -\int_1^t\iint_{\R^2} e^{is\Phi} \dfrac{s\partial_\xi\Phi}{\Phi} \chi_3(\vec{\xi}) \, \partial_s[\phi_*(\xi_1)\phi_*(\xi_2)\phi_{hi}(\xi_3)]
  \who{f}(\xi_1)\who{f}(\xi_2)\who{f}(\xi_3)\ddd
\label{eq:Ju112}	
\\
& -\int_1^t\iint_{\R^2} e^{is\Phi} \dfrac{\partial_\xi\Phi}{\Phi} \chi_3(\vec{\xi})\, \who{f_*}(\xi_1)\who{f_*}(\xi_2)\who{f_{hi}}(\xi_3)\ddd
\label{eq:Ju113}	
\\
& -\int_1^t\iint_{\R^2} e^{is\Phi} \dfrac{s\partial_\xi\Phi}{\Phi}\chi_3(\vec{\xi}) \, \phi_*(\xi_1)\phi_*(\xi_2)\phi_{hi}(\xi_3)
 \partial_s \big[ \who{f}(\xi_1) \who{f}(\xi_2) \who{f}(\xi_3) \big] \ddd.
\label{eq:Ju114}	
\end{align}


We turn to \eqref{eq:Ju111} and fix $s\in\{1,t\}$.
In the support of the integral we have $|\xi_3| \gtrsim \max\{|\xi_1|,|\xi_2|\}$; thus, recalling \eqref{Phases2} and \eqref{Phases3}, we can write
	\begin{align*}
	\iint_{\R^2}e^{is\Phi} & \dfrac{s\partial_\xi\Phi}{\Phi} \chi_3(\vec{\xi}) \, \who{f_*}(\xi_1)\who{f_*}(\xi_2)\who{f_{hi}}(\xi_3)\dd
	= s \,e^{\frac{is\xi^2}{2}} \F\big(T_m[\bar{u}_*,\bar{u}_*, \bar{u}_{hi}]\big)(s,\xi),
	\end{align*}
where $m$ is a symbol satisfying the hypotheses of Lemma~\ref{lem:tri}(iii); that is, $|\xi_3|m$ is a Coifman--Meyer symbol.  Applying \eqref{TriEst3} with $N \sim s^{-1/2}$ as usual, we get
\begin{align*}
\| \eqref{eq:Ju111}  \|_{L_\xi^2} \lesssim \| u(1) \|_{L_x^\infty}^2 \| \wh{f}(1)\|_{L_\xi^\infty} + t^{\frac54} \| u(t) \|_{L_x^\infty}^2 \| \wh{f}(t)\|_{L_\xi^\infty}.
\end{align*}
In view of \eqref{eq:Xdecay}, this is an acceptable contribution to the right-hand side of \eqref{eq:Ju bootstrap}.


To estimate \eqref{eq:Ju112}, we recall that we can write
$\partial_s\phi_*(\xi_j)\wh{f}(\xi_j) = s^{-1} \wh{f_{med}}(\xi_j)$, where $f_{med}$ denotes the projection of $f$ to frequencies $\sim s^{-1/2}$.  Considering all of the possibilities, one can see that to treat \eqref{eq:Ju112} it ultimately suffices to show how to estimate a term such as
\begin{equation}\label{eq:Ju1121}	
\int_1^t \iint_{\R^2} e^{is\Phi} \dfrac{\partial_\xi\Phi}{\Phi} \chi_3(\vec{\xi}) \, \who{f_*}(\xi_1)\who{f_*}(\xi_2)\who{f_{med}}(\xi_3) \ddd.
\end{equation}
To this end, we write $\eqref{eq:Ju1121} = \int_1^t e^{\frac{is\xi^2}{2}} \F(T_m[\bar{u_*},\bar{u_*},\bar{u_{hi}}]) (s,\xi)\,ds$,
where $m$ is a symbol satisfying the hypotheses of Lemma~\ref{lem:tri}(iii). We now estimate using \eqref{TriEst3}, as we did for \eqref{eq:Ju111} above:
\begin{align*}
\| \eqref{eq:Ju1121}  \|_{L_\xi^2} \lesssim
  \int_1^t s^{\frac14} \| u(s) \|_{L_x^\infty}^2 \| \wh{f}(s)\|_{L_\xi^\infty} \, ds,
\end{align*}
which is an acceptable contribution to the right-hand side of \eqref{eq:Ju bootstrap}.
Note that \eqref{eq:Ju113} is a term of the same type and can be estimated similarly. We skip the details.
	

To estimate \eqref{eq:Ju114}, we once again use the fact that $\partial_\xi\Phi(\Phi)^{-1}\chi_3$ satisfies the hypotheses of Lemma~\ref{lem:tri}(iii), 
together with the identity $e^{is\partial_{xx}/2}\partial_s f = \text{\O}(u^3)$ (cf. \eqref{dsf}); thus \eqref{TriEst3} implies
\begin{align*}
\| \eqref{eq:Ju114}  \|_{L_\xi^2} \lesssim
  \int_1^t s^{\frac54} \| u(s) \|_{L_x^\infty} \| \text{\O}(u^3)(s) \|_{L_x^\infty} \| \wh{f}(s)\|_{L_\xi^\infty} \, ds
  \lesssim \int_1^t s^{\frac54} \| u(s) \|_{L_x^\infty}^4 \| \wh{f}(s)\|_{L_\xi^\infty} \, ds,
\end{align*}
which is acceptable in light of \eqref{eq:Xdecay}.


\subsection{Estimation of \texorpdfstring{\eqref{eq:Ju2}}{Lg}}\label{sec:Ju2}
As before we write
	$$
	1=[\phi_{lo}(\xi_1)+\phi_{hi}(\xi_1)]
	[\phi_{lo}(\xi_2)+\phi_{hi}(\xi_2)]
	[\phi_{lo}(\xi_3)+\phi_{hi}(\xi_3)]
	$$
in the integrand of \eqref{eq:Ju2}.  We estimate the contribution of the term  $\phi_{lo}(\xi_1)\phi_{lo}(\xi_2)\phi_{lo}(\xi_3)$ by volume bounds, as in \eqref{eq:Ju10}.

For the remaining terms we proceed as we did in Section~\ref{sec:fhat2} and write
$1=\chi_1(\vec{\xi})+\chi_2(\vec{\xi})+\chi_3(\vec{\xi})$ for $\vec{\xi}\in\R^3,$
where each $\chi_j$ is a smooth Coifman--Meyer multiplier such that \eqref{eq:max specified0} holds.
As before, we will show how to estimate the contribution of $\chi_2$; similar ideas suffice to treat the contribution of $\chi_1$ and $\chi_3$ (see Remark~\ref{remark:Ju} below). As before, we only need to consider the contribution of $\chi_2$ in terms containing $\wh{f}_{hi}(\xi_2)$,
since if $\max_j\vert\xi_j\vert\lesssim s^{-\frac12}$, then we can simply estimate by volume bounds, as we did for \eqref{eq:Ju10}.

On the support of $\chi_2(\vec{\xi})$ we further decompose $1=
	\chi_\eta(\vec{\xi})+
	\chi_\sigma(\vec{\xi})+
	\chi_s(\vec{\xi}),$
as we did in Section~\ref{sec:fhat2}; see \eqref{eq:chi-eta}--\eqref{eq:chi-s}. Once again we employ the notation $\chi_{2,*}=\chi_2\chi_*$ and  find ourselves faced with estimating the following:
	\begin{align}
	&\int_1^t\iint_{\R^2} e^{is\Psi}
	\chi_{2,\eta}(\vec{\xi})[is\partial_\xi\Psi]
	\wh{f_*}(\xi_1)\wh{f_{hi}}(\xi_2)\wh{f_*}(\xi_3)\ddd,
	\label{eq:Ju21}	
	\\
	&\int_1^t\iint_{\R^2} e^{is\Psi}
	\chi_{2,\sigma}(\vec{\xi})[is\partial_\xi\Psi]
	\wh{f_*}(\xi_1)\wh{f_{hi}}(\xi_2)\wh{f_*}(\xi_3)\ddd,
	\label{eq:Ju22}	
	\\
	&\int_1^t\iint_{\R^2} e^{is\Psi}
	\chi_{2,s}(\vec{\xi})[is\partial_\xi\Psi]
	\wh{f_*}(\xi_1)\wh{f_{hi}}(\xi_2)\wh{f_*}(\xi_3)\ddd.
	\label{eq:Ju23}	
	\end{align}


\subsubsection{Estimation of \eqref{eq:Ju21}}
On the support of $\chi_{2,\eta}$ we have
	\begin{equation}
	\label{eq:Ju21 lb}	
	\vert\partial_\eta\Psi\vert
	=\vert\xi_2-\xi_1\vert\gtrsim\vert\xi_2\vert
	\gtrsim\vert \vec{\xi}\vert.
	\end{equation}
Thus we can use the identity $e^{is\Psi}= (is\partial_\eta\Psi)^{-1} \partial_\eta e^{is\Psi}$ and integrate by parts:
	\begin{align}
	\eqref{eq:Ju21} & =
	-\int_1^t\iint_{\R^2} e^{is\Psi}\partial_\eta \big(\dfrac{\partial_\xi\Psi}{\partial_\eta\Psi}\big) \chi_{2,\eta}(\vec{\xi})
	\wh{f_*}(\xi_1)\wh{f_{hi}}(\xi_2)\wh{f_*}(\xi_3)\ddd \label{eq:Ju211}	
	\\
	& \quad -\int_1^t\!\iint_{\R^2} e^{is\Psi} \dfrac{\partial_\eta[\chi_{2,\eta}(\vec{\xi}) \phi_*(\xi_1)\phi_{hi}(\xi_2)]}{\partial_\eta\Psi}
	[\partial_\xi\Psi] \wh{f}(\xi_1)\wh{f}(\xi_2)\wh{f_*}(\xi_3)\ddd
	\label{eq:Ju212}	
	\\
	& \quad-\int_1^t\iint_{\R^2} e^{is\Psi}\chi_{2,\eta}(\vec{\xi}) \dfrac{\partial_\xi\Psi}{\partial_\eta\Psi}
	\phi_*(\xi_1)\partial_\eta\wh{f}(\xi_1) \wh{f_{hi}}(\xi_2)\wh{f_*}(\xi_3)\ddd \label{eq:Ju214}	
	\\
	& \quad-\int_1^t\iint_{\R^2} e^{is\Psi}\chi_{2,\eta}(\vec{\xi})\dfrac{\partial_\xi\Psi}
	{\partial_\eta\Psi}\wh{f_*}(\xi_1)\phi_{hi}(\xi_2) \partial_\eta\wh{f}(\xi_2)\wh{f_*}(\xi_3)\ddd. \label{eq:Ju215}	
	\end{align}

We first consider \eqref{eq:Ju211}.  In the support of the integral, we have $|\xi_2|\gtrsim\max\{|\xi_1|,|\xi_3|\}$; thus, recalling \eqref{Phases2}, \eqref{Phases3},  and \eqref{eq:Ju21 lb}, we can write
$ \eqref{eq:Ju211} = \int_1^t e^{\frac{is\xi^2}{2}} \F\big(T_m[u_*,u_{hi},u_*]\big)(s,\xi)\,ds,$
where $m = \partial_\eta (\partial_\xi\Psi/\partial_\eta\Psi)\chi_{2,\eta}$ satisfies the hypotheses of Lemma~\ref{lem:tri}(iii).
We can then apply \eqref{TriEst3}, as we did for the term \eqref{eq:Ju1121} above, to obtain an acceptable bound. As in Section~\ref{sec:fhat2}, throughout Section~\ref{sec:Ju2} we apply Lemma~\ref{lem:tri} with $\xi_2$ playing the role of $\xi_3$.


We next turn to \eqref{eq:Ju212}.
As before, $\partial_\eta\phi_*(\xi_j) = \pm s^{\frac12}\phi'(s^{\frac12}\xi_j)$,
and multiplication by $\phi'(s^{\frac12}\cdot)$ corresponds to a projection to frequencies $\sim s^{-\frac12}$, which we denote by $P_{med} f = f_{med}$.  Considering all of the possibilities, one finds that to treat \eqref{eq:Ju212}, it suffices to estimate
\begin{align}
\label{eq:Ju2121}
& \int_1^t\iint_{\R^2} \, e^{is\Psi} \dfrac{\partial_\eta \chi_{2,\eta}(\vec{\xi})}{\partial_\eta\Psi}[\partial_\xi\Psi]
  \, \wh{f_*}(\xi_1) \wh{f_{hi}}(\xi_2) \wh{f_*}(\xi_3)\ddd,
\\
\label{eq:Ju2122}
& \int_1^t\iint_{\R^2} s^{\frac12} \, e^{is\Psi} \dfrac{\chi_{2,\eta}(\vec{\xi})}{\partial_\eta\Psi}[\partial_\xi\Psi]
  \, \wh{f_*}(\xi_1) \wh{f_{med}}(\xi_2) \wh{f_*}(\xi_3)\ddd.
\end{align}

Using \eqref{eq:Ju21 lb} and the fact that $\xi_2 \partial_\eta \chi_{2,\eta}$ is Coifman--Meyer,
we can write
\[
\eqref{eq:Ju2121} = \int_1^t e^{\frac{is\xi^2}{2}} \F\big(T_m[u_*,u_{hi},u_*]\big)(s,\xi)\,ds,
\]
where $m$ satisfies the hypotheses of Lemma~\ref{lem:tri}(iii).  In particular, this term can be treated as \eqref{eq:Ju211} above.  We next write
$\eqref{eq:Ju2122} = \int_1^t s^{\frac12} e^{\frac{is\xi^2}{2}} \F\big(T_m[u_*,u_{med},u_*]\big)(s,\xi)\,ds$, where $m$ is a Coifman--Meyer symbol (cf. \eqref{eq:Ju21 lb}).  Using Lemma~\ref{thm:cm} and \eqref{est0}, we get an acceptable bound:
\begin{align*}
\| \eqref{eq:Ju2122} \|_{L_x^2} \lesssim \int_1^t s^{\frac12} \|u\|_{L_x^\infty}^2 \|u_{med}\|_{L_x^2}\,ds
  \lesssim \int_1^t s^{\frac14}\|u\|_{L_x^\infty}^2 \|\wh{f}\|_{L_\xi^\infty}\,ds.
\end{align*}
	
We next write $\eqref{eq:Ju214} = \int_1^t e^{\frac{is\xi^2}{2}}\F\big(T_m[P_* Ju,u_{hi},u_*]\big)(s,\xi)\,ds$,
where $m =\chi_{2,\eta}\partial_\xi\Psi/\partial_\eta\Psi$.  Recalling \eqref{Phases3} and \eqref{eq:Ju21 lb}, we see that $m$ is Coifman--Meyer; thus, Lemma~\ref{thm:cm} gives
\begin{align*}
\|\eqref{eq:Ju214}\|_{L_\xi^2} & \lesssim \int_1^t \|u(s)\|_{L_x^\infty}^2 \|Ju(s)\|_{L_x^2}\,ds,
\end{align*}
which is acceptable.  As \eqref{eq:Ju215} can be estimated similarly, we complete the treatment of \eqref{eq:Ju21}.


\subsubsection{Estimation of \eqref{eq:Ju22}}
This term is very similar to \eqref{eq:Ju21}. In particular, on the support of $\chi_{2,\sigma}$ we have
	$\vert\partial_\sigma\Psi\vert=\vert\xi_2-\xi_3\vert
	\gtrsim\vert\xi_2\vert\gtrsim\vert\vec{\xi}\vert.$
Thus we can use the identity $
	e^{is\Psi} =  (is\partial_\sigma\Psi)^{-1}\partial_\sigma e^{is\Psi}$
to integrate by parts in $\sigma,$ and the same ideas used to estimate \eqref{eq:Ju21} then suffice to handle the resulting terms.


\subsubsection{Estimation of \eqref{eq:Ju23}}
As in \eqref{eq:fhat2 s lb} we note that on the support of $\chi_{2,s}$ we have
\[
\vert\Psi\vert\gtrsim \vert\xi_2\vert^2\sim\vert\xi_1\vert^2\sim \vert\xi_3\vert^2\sim\vert\vec{\xi}\vert^2.
\]
Thus, we can use the identity $e^{is\Psi} = \partial_s e^{is\Psi} (i\Psi)^{-1}$ and integrate by parts in $s$ to get
	\begin{align}
	\eqref{eq:Ju23}&=
	\bigg[\iint_{\R^2} e^{is\Psi}\chi_{2,s}(\vec{\xi}) \dfrac{s\partial_\xi\Psi}{\Psi}\wh{f_*}(\xi_1)
	\wh{f_{hi}}(\xi_2)\wh{f_*}(\xi_3)\dd\bigg]_{s=1}^t
	\label{eq:Ju231}	
\\
& \quad\!-\!\int_1^t\!\!\!\iint_{\R^2}\!\!
	\dfrac{s \, e^{is\Psi}\chi_{2,s}\!(\vec{\xi})
	\partial_\xi\Psi}{\Psi}  \partial_s[\phi_*(\xi_1)\phi_{hi}(\xi_2)\phi_*(\xi_3)]\,
	\wh{f}(\xi_1)\wh{f}(\xi_2)\wh{f}(\xi_3)\ddd
	\label{eq:Ju232}	
\\
& \quad-\int_1^t \iint_{\R^2} e^{is\Psi}\chi_{2,s}(\vec{\xi})
	\dfrac{\partial_\xi\Psi}{\Psi} \wh{f_*}(\xi_1) \wh{f_{hi}}(\xi_2)\wh{f_*}(\xi_3)\ddd
	\label{eq:Ju233}	
\\
& \quad-\int_1^t\iint_{\R^2} e^{is\Psi}\chi_{2,s}(\vec{\xi})
	\dfrac{s\partial_\xi\Psi}{\Psi}\phi_*(\xi_1)\phi_{hi}(\xi_2)\phi_*(\xi_3)
	\partial_s\big[ \wh{f}(\xi_1) \wh{f}(\xi_2)\wh{f}(\xi_3) \big] \ddd.
	\label{eq:Ju234}	
\end{align}

Thanks to the lower bound on $\Psi$, these terms are similar to the ones in \eqref{eq:Ju111}--\eqref{eq:Ju114}.
In fact,
\eqref{eq:Ju231} can be estimated exactly like the term \eqref{eq:Ju111}.
For the term \eqref{eq:Ju232}, we can argue as in the estimate of \eqref{eq:Ju112} (see also \eqref{eq:Ju1121}).  Furthermore, the term \eqref{eq:Ju233} is similar to \eqref{eq:Ju113}, while \eqref{eq:Ju234} is similar to \eqref{eq:Ju114}.
In particular, applying the trilinear estimate \eqref{TriEst3} in each case leads to acceptable contributions.


\begin{remark}\label{remark:Ju}
We have estimated \eqref{eq:Ju21}--\eqref{eq:Ju23},
which completes the estimation of the contribution of $\chi_2$ to \eqref{eq:Ju2}.
As in Remark~\ref{remark:fhat}, we can also decompose the support of $\chi_1$ and $\chi_3$
so that we have suitable lower bounds for $\partial_\eta\Psi$, $\partial_\sigma\Psi$, or $\Psi$.
Thus, we can use similar ideas as above to estimate the contribution of $\chi_1$ and $\chi_3$.
\end{remark}
	


\subsection{Estimation of \texorpdfstring{\eqref{eq:Ju3}}{Lg}}
We can estimate \eqref{eq:Ju3} in a very similar manner to \eqref{eq:Ju2}.
Once again the heart of matter is to decompose frequency space (away from the origin) into regions where
one has suitable lower bounds on either the phase $\Omega$ or its derivatives.
See Section~\ref{sec:fhat bootstrap3} for a detailed discussion of this decomposition.

\subsection{Estimation of \texorpdfstring{\eqref{eq:Ju gauge}}{Lg}}
We can handle \eqref{eq:Ju gauge} quite simply thanks to the gauge-invariance of $\vert u\vert^2 u$.
Indeed, using \eqref{eq:factorization} we can rewrite
	$$\eqref{eq:Ju gauge}=\int_1^t s^{-1}\partial_\xi\F\bar{M}(s)\F^{-1}
	\big(\vert\F\M f\vert^2\F\M f\big)(s,\xi)\,ds.$$
Noting that $\partial_\xi\F\bar{M}\F^{-1}=\F\bar{M}\F^{-1}\partial_\xi$ and using \eqref{gauge X} and \eqref{eq:forms of Ju}, we can estimate
	$$
	\|\eqref{eq:Ju gauge}\|_{L_\xi^2}\lesssim \int_1^t s^{-1}\|\F\M f(s)\|_{L_\xi^\infty}^2\|\partial_\xi\F\M f(s)\|_{L_x^2}\,ds
	\lesssim \int_1^t s^{-\frac34}\|u(s)\|_{X}^3\,ds,
	$$	
which is acceptable. This completes 
the proof of Proposition~\ref{prop:Ju}.

\section{Norm growth for a model nonlinearity}\label{sec:grow}

In this section we study the model equation
\begin{equation}\tag{\ref{model}}
(i\partial_t+\tfrac12\partial_{xx}) u = i|u|^2 u
\end{equation}
and prove Theorem~\ref{thm:grow}.  Throughout the section, we suppose $u$ is a solution to \eqref{model} as in the statement of Theorem~\ref{thm:grow},
with $f(t)=e^{-it\partial_{xx}/2}u(t)$. 
In particular, $\|u_1\|_{\Sigma}=\eps$, $u$ is defined at least up to time $T_\eps=\exp(\frac{1}{c\eps^2})$ for some $c>0$, 
and $u$ satisfies the bounds given in \eqref{worse-J}.  We write $T_{\max}\in(T_{\eps},\infty]$ for the maximal time of existence of $u$.

The plan is to exhibit growth in time of $|\wh{f}(t,\xi)|^2$ by comparing it to a (growing) solution to an ODE (cf. \eqref{ode} and \eqref{approx-ode} below).  To prove that the ODE accurately models the PDE requires good bounds for the solution.  One of the benefits of working with \eqref{model} is that we can prove a better estimate for the $L_x^2$-norm of $Ju$ than the one given in \eqref{worse-J}.  (Recall that the bound in \eqref{worse-J} holds with an arbitrary cubic nonlinearity.)  In particular, we have the following.

\begin{lemma}[Improved bounds]\label{L:improved bounds} If $\eps>0$ is sufficiently small, then
\begin{equation}\label{good-estimates}
\sup_{t\in[1,T_{\eps}]} \bigl\{ \|\wh{f}(t)\|_{L_\xi^\infty} + t^{\frac12}\|u(t)\|_{L_x^\infty} + t^{-\frac{1}{10}}\|u(t)\|_{L_x^2} + t^{-\frac{1}{10}}\|Ju(t)\|_{L_x^2}\bigr\} \lesssim \eps.
\end{equation}
\end{lemma}

\begin{proof} Comparing with \eqref{worse-J}, we need only consider the $L_x^2$-norms.

Direct computation shows $J(t):=x+it\partial_x=M(t) it\partial_x\bar{M}(t)$, where $M(t)=e^{ix^2/2t}$.  Thus,
\[
\|J\bigl(|u|^2 u)\|_{L_x^2} \sim \| t\partial_x (|\bar M u|^2 \bar M u)\|_{L_x^2} \lesssim \|u\|_{L_x^\infty}^2 \|Ju\|_{L_x^2},
\]
and hence by the Duhamel formula and \eqref{worse-J},
\[
\|Ju(t)\|_{L_x^2} \leq C\eps + [C\eps]^2\int_1^t \|Ju(s)\|_{L_x^2} \tfrac{ds}{s}.
\]
Thus, by Gronwall's inequality, we have
$\|Ju(t)\|_{L_x^2} \lesssim t^{[C\eps]^2}\eps$ for $1\leq t\leq T_\eps$, which suffices if $\eps$ is small enough.  The same argument treats  the $L_x^2$-norm of $u$.
\end{proof}

Next, we prove that we can propagate bounds for $u$ as long as we can control $\wh{f}$ in $L_\xi^\infty$.

\begin{lemma}[Propagating bounds]\label{L:propagate} Suppose $T_{\eps}\leq T_1<T_2<T_{\max}$ and
\[
 \|u(T_1)\|_{L_x^2} + \|Ju(T_1)\|_{L_x^2}+ \sup_{t\in[T_1,T_2]}\|\wh{f}(t)\|_{L_\xi^\infty}  \leq \mu
\]
for some $\mu>0$.  If $\mu$ is sufficiently small, then
\[
\sup_{t\in[T_1,T_2]} \bigl\{t^{-\frac{1}{10}} \|u(t)\|_{L_x^2} + t^{-\frac{1}{10}} \|Ju(t)\|_{L_x^2}\bigr\} \lesssim \mu.
\]
\end{lemma}

\begin{proof}
The proof is similar to the arguments above. Define the set
\[
S=\{t\in[T_1,T_2]: t^{-\frac{1}{10}}\|Ju(t)\|_{L_x^2} < C\mu\}.
\]
By assumption, $T_1\in S$ for some appropriate choice of $C$.
Suppose toward a contradiction that $S\neq[T_1,T_2]$. By continuity, we can find a first time $T\in(T_1,T_2]$ so that
\begin{equation}\label{bs-lb}
\|Ju(T)\|_{L_x^2} \geq C\mu  T^{\frac{1}{10}}.
\end{equation}
Using Lemma~\ref{lem:decay}, we find that
\[
\sup_{t\in[T_1,T]} t^{\frac12} \|u(t)\|_{L_x^\infty} \leq \tilde C\cdot C\mu
\]
for some absolute constant $\tilde C$.  Arguing as in Lemma~\ref{L:improved bounds}, we deduce
\[
\|Ju(T)\|_{L_x^2} \leq C\mu T^{[\tilde C\cdot C\mu]^2}
\]
However, this contradicts \eqref{bs-lb} for $\mu$ small enough. Thus $S=[T_1,T_2]$.  A similar Gronwall argument yields the bounds for the $L^2$-norm of $u$.
\end{proof}

We turn to estimating the size of $\wh{f}$.  We define
\[
A(t,\xi):= 2|\wh{f}(t,\xi)|^2
\]
and observe that for each $\xi\in\R$, the function $A(t,\xi)$ satisfies an ODE in $t$. 
Indeed, rewriting the equation \eqref{model} as $\partial_t f= e^{-it\partial_{xx}/2}(|u|^2 u)$ and using \eqref{eq:factorization}, we deduce
\begin{equation}\label{approx-ode}
\partial_t A = t^{-1} A^2 + t^{-1} R,
\end{equation}
where the remainder $R$ is given by
\begin{equation}\label{def-R}
R= 4\Re\bigl\{\overline{\F f}\bigl[ (\F\bar M\F^{-1} - 1)|\F M f|^2 \F M f + |\F M f|^2 \F M f - |\wh{f}|^2 \wh{f}\ \bigr]\bigr\}.
\end{equation}
We expect that as long as $u$ obeys good estimates, the remainder $R$ will decay in time. Thus the behavior of $A$ should be governed by a (growing) solution to the ODE
\begin{equation}\label{ode}
\partial_t B = t^{-1} B^2.
\end{equation}

We first consider the issue of controlling the remainder.

\begin{lemma}[Controlling the remainder]\label{L:remainder} Suppose $1\leq T_1<T_2<T_{\max}$ and
\begin{equation}\label{remainder1}
\sup_{t\in[T_1,T_2]} \bigl\{\| \wh{f}(t)\|_{L_\xi^\infty} + t^{-\frac{1}{10}} \| Ju(t)\|_{L_x^2} + t^{-\frac{1}{10}} \|u(t)\|_{L_x^2} \bigr\} \leq \mu
\end{equation}
for some $\mu>0$. Then
\[
\sup_{t\in[T_1,T_2]} t^{\frac1{10}} \|R(t)\|_{L_\xi^\infty} \lesssim \mu^4.
\]
\end{lemma}

\begin{proof}  The main ideas appear already in the proof of Lemma~\ref{lem:decay}, but we include the details for completeness. First, note the pointwise bound
$|M(t)-1| \lesssim t^{-\delta} |x|^{2\delta}$ for any $0\leq\delta\leq\tfrac12$.
Taking $\delta=\frac15$, this together with Hausdorff--Young, Cauchy--Schwarz, \eqref{eq:forms of Ju} and \eqref{remainder1} implies
\[
\| \F[M-1] f\|_{L_\xi^\infty} \lesssim t^{-\frac15} \|\langle x\rangle f\|_{L_x^2}  \lesssim t^{-\frac1{10}} \mu.
\]
Using \eqref{remainder1} we also have
$\|\F Mf\|_{L_\xi^\infty}
\lesssim \mu$.

Estimating as above and using Plancherel, we obtain
\begin{align*}
\| \F[\bar M-1]\F^{-1}&\bigl(|\F M f|^2 \F M f\bigr)\|_{L_\xi^\infty} \lesssim t^{-\frac15} \|\langle x\rangle \F^{-1}\bigl(|\F Mf|^2 \F Mf\bigr)\|_{L_x^2}\\
 &\lesssim t^{-\frac15}\|\F Mf\|_{L_x^\infty}^2 \|\langle \partial_x \rangle \F Mf\|_{L_x^2} \lesssim t^{-\frac1{10}}\mu^3.
\end{align*}
Furthermore,
\[
\| |\F M f|^2 \F Mf - |\wh f|^2 \wh f\|_{L_\xi^\infty} \lesssim \mu^2 \|\F[M-1]f\|_{L_\xi^\infty} \lesssim t^{-\frac{1}{10}}\mu^3,
\]
and the result follows.
\end{proof}

We next look for a point $\xi_0\in\R$ such that $A(t,\xi_0) \gtrsim \eps^2$ at $t=T_\eps$.  Using \eqref{approx-ode},
\[
\partial_t\bigl[ A(t)\exp\bigl(-\textstyle\int_1^t A(s)\tfrac{ds}{s}\bigr)\bigr] = t^{-1}\exp\bigl(-\int_1^t A(s)\tfrac{ds}{s}\bigr) R(t).
\]
Thus, exploiting $A\geq 0$, Lemma~\ref{L:improved bounds} and Lemma~\ref{L:remainder} (with $\mu\sim\eps$), we can deduce
\[
A(T_\eps,\xi)\geq A(1,\xi)-\int_1^{T_\eps} \exp\biggl(-\int_1^s A(\sigma)\tfrac{d\sigma}{\sigma}\biggr)R(s)\tfrac{ds}{s}\geq A(1,\xi)-O(\eps^4).
\]
Recalling \eqref{asmp:u1} and taking $\eps$ small, we can find $\xi_0\in\R$ so that
\begin{equation}\label{pw-lb}
A_0:= A(T_\eps,\xi_0) \geq \tfrac{1}{5}\eps^2.
\end{equation}

We now define $B(t)$ to be the solution to \eqref{ode} that agrees with $A(t,\xi_0)$ at time $t=T_\eps$:
\begin{equation}\label{def:B}
B(t):= \frac{A_0}{1-A_0\log(\frac{t}{T_\eps})}, \quad  B:[T_\eps,T_\eps\exp(\tfrac{1}{A_0}))\to[A_0,\infty).
\end{equation}
Note that $B(t)\to\infty$ as $t\to T_\eps\exp(\tfrac{1}{A_0})$.
To show that $B(t)$ is good approximation to $A(t,\xi_0)$ for $t\geq T_\eps$ we introduce
\[
D(t):=A(t,\xi_0)-B(t), \quad D:\bigl[T_\eps,T_{\max}\wedge T_{\eps}\exp(\tfrac{1}{A_0})\bigr)\to\R.
\]
Here and below, $a\wedge b$ denotes $\min\{a,b\}$. We now consider the issue of controlling this difference.

\begin{lemma}[Controlling the difference]\label{L:difference} Suppose $T_{\eps}\leq T\leq T_{\max}\wedge T_{\eps}\exp(\tfrac{1}{A_0})$ and
\begin{equation}\label{bs-d1}
\sup_{[T_\eps,T]}\|\wh{f}(t)\|_{L_\xi^\infty} \leq \mu
\end{equation}
for some $0<\eps\ll \mu\ll1 $. Then
\[
|D(T)| \lesssim \tfrac{1}{A_0} \cdot \bigl(\tfrac{T}{T_\eps}\bigr)^{2\mu^2}\cdot \mu^4T_\eps^{-\frac{1}{10}}\cdot B(T).
\]

\end{lemma}
\begin{proof} Note that $D$ solves
\[
\partial_t D(t) = t^{-1}\bigl(A(t,\xi_0)+B(t)\bigr)D(t) + t^{-1}R(t,\xi_0),\quad D(T_\eps)=0.
\]
Using the integrating factor
\[
\rho(t) := \int_{T_\eps}^t [A(s,\xi_0)+B(s)]\tfrac{ds}{s},
\]
we find
\[
D(t) = e^{\rho(t)}\int_{T_\eps}^t e^{-\rho(s)}R(s,\xi_0)\tfrac{ds}{s}=\int_{T_\eps}^t \exp\biggl(\int_s^t [A(\sigma,\xi_0)+B(\sigma)]\tfrac{d\sigma}{\sigma} \biggr)R(s,\xi_0)\tfrac{ds}{s}.
\]

An explicit computation using \eqref{def:B} shows
\[
\exp\bigl(\textstyle\int_s^T B(\sigma)\tfrac{d\sigma}{\sigma}\bigr) \leq \tfrac{B(T)}{A_0}\quad\text{for all}\quad T_\eps\leq s\leq T<T_\eps\exp(\tfrac{1}{A_0}).
\]

Using \eqref{bs-d1}, we can also estimate
\[
\exp\bigl(\textstyle\int_s^T A(\sigma,\xi_0)\tfrac{d\sigma}{\sigma}\bigr)\leq \bigl(\tfrac{T}{T_\eps}\bigr)^{2\mu^2} \quad\text{for all} \quad T_\eps\leq s\leq T <T_{\max}.
\]

In view of \eqref{bs-d1}, Lemma \ref{L:propagate} and Lemma \ref{L:remainder} we have
\[
\sup_{T_{\eps}\leq t\leq T} t^{\frac{1}{10}}\|R(t)\|_{L_\xi^\infty} \lesssim \mu^4,\quad\text{whence}\quad
\int_{T_\eps}^T |R(s,\xi_0)|\,\tfrac{ds}{s} \lesssim T_{\eps}^{-\frac{1}{10}} \mu^4.
\]
Combining these estimates yields the desired conclusion.
\end{proof}

We now complete the proof of the theorem.

\begin{proof}[Proof of Theorem~\ref{thm:grow}] Let $K\gg \eps^2$ be a constant to be determined below.
We define $T_K$ to be the time such that $B(T_K) = 4K$:
\begin{equation}\label{def:TK}
T_K = T_\eps\exp(\tfrac{1}{A_0}-\tfrac{1}{4K}) \leq T_\eps\exp(\tfrac{5}{\eps^2}-\tfrac{1}{4K}).
\end{equation}
If $T_{\max}\leq T_K$, the conclusion of the theorem holds.  Thus it remains to consider the case $T_{\max}>T_K$, in which case it suffices to show
\[
\| \wh{f}(t)\|_{L_\xi^\infty} \geq K^{\frac12} \quad\text{for some}\quad t\in[T_\eps,T_K].
\]
We proceed by contradiction and suppose that
\begin{equation}\label{asmp:grow1}
\sup_{t\in[T_\eps,T_K]} \|\wh{f}(t)\|_{L_\xi^\infty} \leq K^{\frac12}. 
\end{equation}
Applying Lemma~\ref{L:difference} (with $\mu=K^{\frac12}$), we deduce that
\[
|D(T_K)| \lesssim \tfrac{1}{A_0} \cdot \bigl(\tfrac{T_K}{T_\eps}\bigr)^{2K}\cdot K^2 T_\eps^{-\frac{1}{10}} B(T_K)
\]
Recalling $A_0 \geq \tfrac{1}{5}\eps^2$ and \eqref{def:TK} and rearranging, we find
\[
|D(T_K)| \lesssim \tfrac{1}{\eps^2} \exp(-[\tfrac{1}{10c} - 10K]\tfrac{1}{\eps^2}) K^2 B(T_K).
\]

We now choose $K = \frac{1}{200c}$, so that the above becomes
\[
|D(T_K)| \lesssim \tfrac{1}{\eps^2} \exp(-\tfrac{1}{20c\eps^2}) B(T_K).
\]
For $\eps$ sufficiently small (depending only on the absolute constant $c$), this yields
\[
|D(T_K)| <\tfrac12 B(T_K),\quad\text{whence}\quad
|\wh{f}(T_K,\xi_0)|^2=\tfrac12 A(T_K,\xi_0)> \tfrac14 B(T_K) = K.
\]
This contradicts \eqref{asmp:grow1} and completes the proof of Theorem~\ref{thm:grow}.
\end{proof}

		
\appendix
\section{Construction of cutoffs}\label{section:appendix}
In this section we construct the cutoff functions used in Sections~\ref{sec:fhat}~and~\ref{sec:Ju}.
Recall the notation from \eqref{notfreq}, and let us first describe how to write $1 = \chi_1 + \chi_2 + \chi_3$ as in \eqref{eq:max specified0},
that is, in such a way that
\begin{equation}
\label{notfreqapp}
|\xi_j| \geq \max \big\{\tfrac{9}{10} |\xi_k| \, , \, k=1,2,3 \big\} \quad \text{for all} \quad \vec{\xi} \in\text{support}(\chi_j).
\end{equation}
We let $a$ denote a smooth even function such that $a(x)=1$ for $\vert x\vert\leq 1$ and $a(x)=0$ for $\vert x\vert>1+\delta$
for some small $\delta>0$. 
Denoting $a^c = 1-a$, which is a function supported on $|x| \geq 1$, we let
\begin{align}
\chi_1(\vec{\xi}) := a^c(\tfrac{\xi_1}{\xi_2})a^c(\tfrac{\xi_1}{\xi_3}),
\quad\,
\chi_2(\vec{\xi}) := a(\tfrac{\xi_1}{\xi_2})a^c(\tfrac{\xi_2}{\xi_3}),
\quad\,
\chi_3(\vec{\xi}) := a^c(\tfrac{\xi_1}{\xi_2})a(\tfrac{\xi_1}{\xi_3}) + a(\tfrac{\xi_1}{\xi_2})a(\tfrac{\xi_2}{\xi_3}).
\end{align}
 It is clear that these satisfy the desired property \eqref{notfreqapp}.
Furthermore, as the derivatives of $a$ are supported near $|x|=1$, all of the multipliers appearing above are Coifman--Meyer multipliers satisfying \eqref{CM}.
Also notice that $\xi_j \partial_{\eta} \chi_j(\vec{\xi})$ and $\xi_j \partial_{\sigma} \chi_j(\vec{\xi})$, $j=1,2,3$, are Coifman--Meyer multipliers as well.

Next we describe how to write $1=\chi_\eta+\chi_\sigma+\chi_s$ as in \eqref{eq:chi-eta}--\eqref{eq:chi-s} on the support of $\chi_2$.
We let $b$ be a smooth even function such that $b(x)=1$ for $\vert x\vert\leq\frac{1}{100}$ and $b(x)=0$ for $\vert x\vert>\frac{1}{50}$,
and define
$$
\chi_\eta = 1-b(\tfrac{\xi_2-\xi_1}{\xi_2}),
  \quad \chi_\sigma=b(\tfrac{\xi_2-\xi_1}{\xi_2})\big[1-b(\tfrac{\xi_2-\xi_3}{\xi_2})\big],
  \quad \chi_s=b(\tfrac{\xi_2-\xi_1}{\xi_2})b(\tfrac{\xi_2-\xi_3}{\xi_2}).
$$
Then the inequalities \eqref{eq:chi-eta}--\eqref{eq:chi-s} clearly hold,
and furthermore one can check that the functions $\chi_2\chi_*$ define Coifman--Meyer multipliers.
Finally notice that in view of the support properties of $\chi_2$, the multipliers
$\xi_2 \partial_\eta \big[ \chi_2 \chi_\eta (\vec{\xi}) \big]$ and $\xi_2 \partial_\sigma \big[ \chi_2 \chi_\eta (\vec{\xi}) \big]$
are Coifman--Meyer, as well.


\end{document}